\documentclass[11pt,a4paper]{article}
\usepackage[T1]{fontenc}
\usepackage[utf8]{inputenc}
\usepackage{mathptmx}
\usepackage{amsmath,amssymb,amsthm,mathtools}
\usepackage{enumitem,booktabs,longtable,array,geometry,graphicx}
\usepackage{microtype,caption,placeins,needspace}
\usepackage{titling}
\predate{}\postdate{}
\usepackage{xcolor,tikz}
\usetikzlibrary{calc,positioning,arrows.meta,backgrounds}
\usepackage[hidelinks,bookmarksnumbered]{hyperref}
\theoremstyle{plain}
\newtheorem{theorem}{Theorem}[section]
\newtheorem{lemma}[theorem]{Lemma}
\newtheorem{proposition}[theorem]{Proposition}
\newtheorem{corollary}[theorem]{Corollary}

\theoremstyle{definition}
\newtheorem{definition}[theorem]{Definition}
\theoremstyle{remark}

\newtheorem{example}[theorem]{Example}

\newcommand{\sidx}{\chi'_s}

\DeclareMathOperator{\per}{per}

\DeclareMathOperator{\sgn}{sgn}
\numberwithin{equation}{section}
\setlist{itemsep=2pt,topsep=4pt}
\hypersetup{
 pdftitle={Strong edge-colourings of truncations of cubic graphs},
 pdfauthor={Domenico Labbate and Federico Romaniello},
 pdfsubject={Constructive proofs for bipartite and claw-free cubic bases},
 pdfkeywords={strong edge-colouring, graph truncation, cubic graphs, bipartite graphs, claw-free graphs, det-extremal graphs}
}
\title{\bfseries Strong edge-colourings of truncations of cubic graphs}
\preauthor{}\postauthor{}
\newcommand{\paperauthor}[2]{%
 \begin{center}
 {\large #1}\\[1mm]
 {\normalsize\itshape #2}
 \end{center}%
}
\author{%
\paperauthor{Domenico Labbate}{%
Dipartimento di Ingegneria,\\
Universit\`a degli Studi della Basilicata, Potenza, Italy}
\paperauthor{Federico Romaniello}{%
Dipartimento per l'Innovazione Umanistica, Scientifica e Sociale (DIUSS),\\
Universit\`a degli Studi della Basilicata, Potenza, Italy}
}
\date{}

\begin{document}
\maketitle
\begingroup
\renewcommand{\thefootnote}{}
\footnotetext[0]{\raggedright\textit{E-mail addresses:}
\href{mailto:domenico.labbate@unibas.it}{domenico.labbate@unibas.it}
(Domenico~Labbate),
\href{mailto:federico.romaniello@unibas.it}{federico.romaniello@unibas.it}
(Federico~Romaniello).}
\endgroup

\begin{abstract}
We study a problem posed by Kardo\v{s} asking whether the truncation
of every connected simple bridgeless cubic graph has strong chromatic
index six. The truncation $T(G)$ is obtained by replacing each vertex
of $G$ with a triangle and attaching its incident edges to distinct
vertices of that triangle. An affirmative answer was previously
established for prism bases by Han and Cui. We prove that
$\sidx(T(G))=6$ whenever $G$ is a connected simple cubic graph
that is bipartite or is both bridgeless and claw-free. Both proofs
are constructive and use a common compatibility criterion for
combining colourings of smaller pieces. The bipartite result also
extends to non-det-extremal multigraphs.

\smallskip
\noindent\textbf{Math. Subj. Class.:} 05C15, 05C70, 05C76.
\end{abstract}

\section{Introduction}\label{intro:section}

All graphs considered in this paper are finite, non-empty and loopless. Unless
specified otherwise, they are simple. Whenever parallel edges are
allowed, these edges and their incidences are distinguished. For a graph
$G$, we denote its vertex set and edge set by $V(G)$ and $E(G)$,
respectively. We write $N_G(v)$
for the neighbourhood of a vertex $v$, $\partial_G(v)$ for the set of
edges incident with $v$, and $\delta_G(S)$ for the set of edges with
exactly one endpoint in $S\subseteq V(G)$. Subscripts are omitted when
there is no ambiguity. A graph is \emph{cubic} if every vertex has
degree three, \emph{subcubic} if its maximum degree is at most three,
and \emph{bridgeless} if no edge is a cut-edge.
The term $k$-connected refers to vertex connectivity; edge connectivity
is always stated explicitly.

A matching is \emph{perfect} if it covers every vertex, and
\emph{induced} if its endpoints induce no edges other than those of
the matching. A cycle is \emph{$M$-alternating} if its edges alternate
between a matching $M$ and its complement. A \emph{strong edge-colouring}
of $G$ is an edge-colouring in which every colour class is an induced
matching. Equivalently, it is a proper edge-colouring in which any two
edges adjacent to a common edge receive distinct colours. Two edges are in \emph{strong conflict} if they share an endpoint or an edge
joins an endpoint of one to an endpoint of the other. The least number
of colours in a strong edge-colouring is the \emph{strong chromatic index}
of $G$, denoted by $\sidx(G)$. Equivalently, any two distinct edges of
$G$ at distance at most two in the line graph $L(G)$ must receive different colours,
so $\sidx(G)=\chi(L(G)^2)$.

Strong edge-colourings were introduced by Fouquet and Jolivet in 1983
\cite{FJ1983}. A central problem is to bound the strong chromatic
index in terms of the maximum degree. In 1985, Erd\H{o}s and
Ne\v{s}et\v{r}il conjectured that every graph of maximum degree $\Delta$
has strong chromatic index at most $5\Delta^2/4$; see
\cite{Erdos1988}. For subcubic graphs, Andersen \cite{Andersen1992}
and, independently, Hor\'ak, He Qing and Trotter \cite{HHT1993}
established the sharp upper bound of ten colours.

Smaller bounds become possible when additional structure is imposed.
The claw-free case is particularly relevant here. A \emph{claw} is
a graph isomorphic to $K_{1,3}$, and a \emph{diamond} is a graph
isomorphic to $K_4$ with one edge deleted. The two nonadjacent vertices
of a diamond are its \emph{tips}. A graph is \emph{claw-free} or
\emph{diamond-free} if it has no induced claw or induced diamond,
respectively.

In 2022, Lv, Li and Zhang \cite{LLZ2022} proved that every connected
claw-free subcubic graph other than the triangular prism admits a
strong edge-colouring with at most eight colours. Lin and Lin
\cite{LinLin2023} subsequently improved this bound to seven and constructed
an infinite family attaining it. These results motivate the study of
more restricted classes in which six colours might suffice.

One such class is provided by truncations. Let $G$ be a loopless cubic
multigraph. Its \emph{truncation} $T(G)$ has a vertex $(v,e)$ for each
incidence $v\in e$. The three vertices associated with $v$ induce a
triangle $\Delta_v$, called a \emph{canonical triangle}. For every
edge $e=uv$ of $G$, the graph $T(G)$ has the edge $(u,e)(v,e)$.
Edges within canonical triangles are called \emph{internal}, and
the edges corresponding to $E(G)$ are called \emph{external}. The
vertex $(v,e)$ is the \emph{port} of $e$ at $v$. In particular,
parallel edges of $G$ have distinct ports, and $T(G)$ is a simple
cubic graph that is both claw-free and diamond-free.

Han and Cui \cite{HanCui2023} proved that truncations of the prisms
$C_n\square K_2$, for every $n\geq3$, admit strong edge-colourings
with six colours. At the 33rd Workshop on Cycles and Colourings in
2025, Kardo\v{s} \cite[Problem~4.1]{Kardos2025} asked whether six
colours suffice for every diamond-free claw-free cubic graph,
also formulating the question for truncations of cubic graphs.
Kardo\v{s} subsequently clarified in a private communication
\cite{KardosPrivate} that the intended formulation includes the
hypothesis of bridgelessness. For simple bases, this leads to the
following question:
\begin{quote}
\emph{Does every connected simple bridgeless cubic graph $G$ satisfy
$\sidx(T(G))=6$?}
\end{quote}

The multigraph formulation requires one further exception. Let $M_0$
denote the cubic multigraph consisting of two vertices joined by
three parallel edges. Its truncation $T(M_0)$ is the triangular prism
and has strong chromatic index nine.

Tanwar \cite{Tanwar2026} constructed an
$18$-vertex diamond-free claw-free cubic graph with strong
chromatic index seven.\footnote{Tanwar uses diamond-free to exclude
diamonds as arbitrary subgraphs. For connected simple cubic graphs,
this convention differs from ours only for $K_4$ and does not affect
his example.} His graph has a bridge, and its triangle
contraction has parallel edges. It therefore does not refute the
bridgeless version of the problem, nor the version restricted to
simple cubic bases.

The present paper establishes two further cases. The three edges of
a canonical triangle and the three external edges incident with it
require six distinct colours. We show that this local lower bound
is attained whenever the simple cubic base is bipartite or is
bridgeless and claw-free. The proofs describe how to choose colourings
on small pieces and make them agree where the pieces meet.

The two main results can now be stated precisely.

\begin{theorem}\label{bip:main}
Let $G$ be a connected simple cubic bipartite graph. Then
\[
 \sidx(T(G))=6.
\]
\end{theorem}

\begin{theorem}\label{cf:main}
Let $G$ be a connected simple bridgeless cubic claw-free
graph. Then
\[
 \sidx(T(G))=6.
\]
\end{theorem}

The construction used to prove Theorem~\ref{cf:main} also gives
a result for double truncations: for every loopless
cubic multigraph $Q$,
\[
 \sidx(T(T(Q)))=6.
\]
This is Corollary~\ref{dt:multigraph-extension}, stated after the
proof of the claw-free theorem.

The common starting point is a local description of the colouring.
Section~\ref{pre:section} identifies the information that must be
recorded at each port and gives the compatibility condition for joining
two coloured pieces. Once that condition is satisfied at every
interface, the local colourings together give a strong colouring of
the whole truncation.

Section~\ref{bip:section} proves Theorem~\ref{bip:main}. Matching signs
separate the argument into two cases. The characterisation of the
non-det-extremal case by mixed binary neighbourhoods is the bipartite
formulation of a result of Choi, Kwon and Park
\cite[Theorem~3.5]{CKP2026}. We give a direct constructive proof and
show how the resulting labels yield an explicit $6$-colour formula.
In the det-extremal case, we combine our rooted replacement
construction and a Heawood colouring with the structural theorems
of McCuaig~\cite{McCuaig2000} and Funk, Jackson, Labbate and
Sheehan~\cite{FJLS2003}.

Section~\ref{cf:section} proves Theorem~\ref{cf:main}. In this
reduction, $G$ denotes the original base, $G_0$ the reduced base,
and $Q$ its skeleton. When $G\ne K_4$ has a triangle block,
suppressing its diamond chains gives $G_0\cong T(Q)$. We construct
a colouring of $T(G_0)$ and then restore the diamonds to colour
$T(G)$, preserving the required boundary signatures. The
double-truncation corollary follows at the end of the section.

Section~\ref{ext:section} proves extensions to multigraph bases
using rooted replacements and blocks formed by parallel edges.
Section~\ref{concl:section} distinguishes these results from the
remaining questions about single truncations.

\FloatBarrier
\section{Local colourings}\label{pre:section}

Let $G$ be a loopless cubic multigraph. The three edges of
any canonical triangle $\Delta_v$ and the three external edges incident
with its ports are pairwise in strong conflict. Consequently,
\begin{equation}\label{pre:lowerbound}
 \sidx(T(G))\geq6.
\end{equation}

To attain this bound, the six edges around each canonical triangle must
use all six colours. The problem is therefore to choose these local
assignments so that they remain compatible across external edges.

Fix a set $\Omega$ of six colours. The labels of these colours may
vary between constructions. A \emph{local state} at $v$ assigns distinct
colours to the three internal edges of $\Delta_v$ and its three
incident external edges. Write $c_v(e)$ for the external colour at the
port of $e\in\partial(v)$. Let $I_v$ be the set of internal colours,
the complement in $\Omega$ of the three external colours. For $e\in\partial(v)$,
let $h_v(e)$ be the colour of the internal edge opposite the port $(v,e)$.
Thus $c_v$ and $h_v$ map $\partial_G(v)$ bijectively onto the external
and internal colour sets, respectively. The two internal colours exposed
at that port form
\[
 P_v(e)=I_v\setminus\{h_v(e)\}.
\]
The pair $(c_v(e),P_v(e))$ is the \emph{signature} at this port.
Thus $h_v(e)$ records the internal colour that is not incident with
the port, whereas $P_v(e)$ records the two internal colours that are.
The signature retains exactly the information needed when this port
is joined to another one.
For an ordered list $(e_0,e_1,e_2)$ of incidences, we may abbreviate the
state as $(c_v(e_0),c_v(e_1),c_v(e_2);h_v(e_0),h_v(e_1),h_v(e_2))$.
In a simple base, a port order may be specified by listing the
neighbours. With parallel edges, we use edge labels whenever needed
to distinguish the incidences.

Later we will replace a vertex or an edge by a larger coloured piece.
To describe the boundary of such a piece, we use \emph{cubic poles}:
loopless multigraphs with dangling half-edges in which every retained
vertex has three incidences, counting half-edges. A dangling incidence
is called a \emph{port} of the pole.

The \emph{augmented truncation} of a pole replaces each retained vertex
by its incidence triangle, keeps the ordinary external edges, and
replaces each dangling half-edge by a pendant edge to a distinct new
degree-$1$ vertex. These auxiliary degree-$1$ vertices are not
truncated. A strong colouring of a pole means a strong colouring of
this augmented truncation. The signature at a port of the pole is
measured at the corresponding retained endpoint in the augmented
truncation.

Joining two dangling half-edges at the base level replaces them by one
ordinary edge. In the augmented truncations, this deletes the two
auxiliary endpoints and replaces their pendant edges by one external
edge joining the retained ports. The operation may join ports of the
same pole or of different poles, provided that the resulting base is
loopless and cubic. The following criterion expresses compatibility
entirely in terms of these signatures.

\begin{lemma}\label{pre:interface}
Let $G$ be a loopless cubic multigraph, let $\Omega$ be a set of six
colours, and let $(c_v,h_v)$ be a local state at each vertex $v\in V(G)$.
These states define a strong edge-colouring of $T(G)$ with colour set
$\Omega$ if and only if, for every edge $e=uv\in E(G)$,
\begin{equation}\label{pre:exposed}
 c_u(e)=c_v(e),\qquad P_u(e)\cap P_v(e)=\varnothing.
\end{equation}
Equivalently, let $c:E(G)\to\Omega$ be a proper edge-colouring and,
for each $v\in V(G)$, let $h_v:\partial_G(v)\to I_v$ be a bijection,
where $I_v=\Omega\setminus\{c(f):f\in\partial_G(v)\}$.
The assignments $c$ and $(h_v)_{v\in V(G)}$ define a strong
edge-colouring of $T(G)$ if and only if
\begin{equation}\label{pre:hidden}
 I_u\cap I_v\subseteq\{h_u(e),h_v(e)\}\qquad(e=uv).
\end{equation}
For the augmented truncation of a cubic pole, the same criterion applies
to its ordinary edges. A dangling edge imposes only the local $6$-colour
condition at its retained endpoint.
\end{lemma}
\begin{proof}
Suppose first that the local states define a strong edge-colouring.
An external edge must receive the same colour at its two ends, so
$c_u(e)=c_v(e)$. Moreover, each internal edge incident with its port
in $\Delta_u$ is in strong conflict with each internal edge incident
with its port in $\Delta_v$. Hence $P_u(e)\cap P_v(e)=\varnothing$,
proving necessity.

Conversely, suppose that~\eqref{pre:exposed} holds for every base
edge. Equality of the external colours makes the colouring
well-defined. Two external edges in strong conflict have ports in
the same canonical triangle, and an external edge conflicts with
internal edges only in its two endpoint triangles. These conflicts,
as well as those between internal edges in the same triangle, are
covered by the local $6$-colour states. Two internal edges in distinct
canonical triangles can be in strong conflict only if an external
edge joins one of their endpoints. Their colours then belong to the
two exposed pairs at that interface, which are disjoint by
\eqref{pre:exposed}. Thus every pair of edges in strong conflict
receives distinct colours, proving sufficiency.

The equivalent formulation~\eqref{pre:hidden} follows from
\[
 P_u(e)\cap P_v(e)
 =(I_u\cap I_v)\setminus\{h_u(e),h_v(e)\}.
\]
For a pole, the auxiliary pendant endpoints introduce no internal
edges, so the same description applies; each dangling edge is covered
by the local state at its retained endpoint.
\end{proof}

In particular, one coloured piece may replace another whenever the
signatures agree at corresponding boundary ports of the two pieces.
This preserves compatibility with the rest of the graph.

Figure~\ref{fig:local-interface} records the boundary data at an external edge.

\begin{figure}[!htbp]
\centering
\begin{tikzpicture}[x=1.25cm,y=1.1cm,
  line width=.65pt,
  vertex/.style={circle,fill=black,inner sep=1.6pt},
  colour/.style={font=\small,fill=white,inner sep=1.5pt},
  every node/.style={font=\small}]
  \coordinate (ul) at (-2.5,1);
  \coordinate (ud) at (-2.5,-1);
  \coordinate (up) at (-1,0);
  \coordinate (vl) at (2.5,1);
  \coordinate (vd) at (2.5,-1);
  \coordinate (vp) at (1,0);
  \draw (ul)--(up) node[midway,above right,colour] {$\alpha_1$};
  \draw (up)--(ud) node[midway,below right,colour] {$\alpha_2$};
  \draw (ul)--(ud) node[midway,left=4pt,colour] {$h_u(e)$};
  \draw (vl)--(vp) node[midway,above left,colour] {$\beta_1$};
  \draw (vp)--(vd) node[midway,below left,colour] {$\beta_2$};
  \draw (vl)--(vd) node[midway,right=4pt,colour] {$h_v(e)$};
  \draw (up)--(vp) node[midway,above=3pt,colour] {$c$};
  \draw (ul)--(-3.3,1.6) (ud)--(-3.3,-1.6)
        (vl)--(3.3,1.6) (vd)--(3.3,-1.6);
  \foreach \p in {ul,ud,up,vl,vd,vp} \node[vertex] at (\p) {};
  \node at (-2,1.62) {$\Delta_u$};
  \node at (2,1.62) {$\Delta_v$};
  \node at (-.75,-.82) {$(c,U)$};
  \node at (.75,-.82) {$(c,V)$};
  \node at (-1.8,-2.03) {$U=\{\alpha_1,\alpha_2\}$};
  \node at (1.8,-2.03) {$V=\{\beta_1,\beta_2\}$};
\end{tikzpicture}
\caption{The boundary data at an external edge corresponding to $e=uv$.
The exposed pairs are $U=\{\alpha_1,\alpha_2\}$ and
$V=\{\beta_1,\beta_2\}$. Their disjointness is the compatibility
condition. The opposite internal colours are $h_u(e)$ and $h_v(e)$.}
\label{fig:local-interface}
\end{figure}
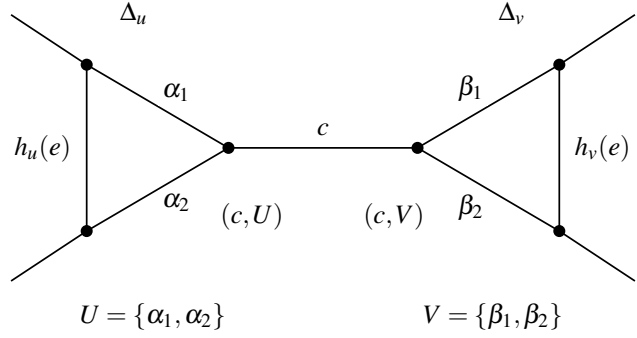

\FloatBarrier
\section{Bipartite cubic graphs}\label{bip:section}

Throughout this section, $G$ is a connected simple cubic bipartite
graph with partition sets $A$ and $B$.
The proof separates according to the signs of the perfect matchings.
When both signs occur, a binary assignment on $A$ gives an explicit
colouring formula. When all signs agree, a rooted version of that
assignment colours a piece with one vertex removed. We first use these
pieces to build the $3$-connected case from the Heawood graph, and then
handle the remaining graphs by cutting and rejoining along $2$-edge cuts.

\subsection{Matching signs and the colouring formula}

Since $G$ is cubic and bipartite, $|A|=|B|$ and Hall's
theorem~\cite{Hall1935} guarantees a perfect matching.
Fix such a matching. Its complement is a union of even cycles.
Colouring the matching edges with one colour and alternating two
further colours on these cycles gives a proper edge-colouring
\[
 \tau:E(G)\longrightarrow\{0,1,2\}.
\]
We refer to the values of $\tau$ as \emph{types}. The same argument
applies to loopless cubic bipartite multigraphs, allowing cycles of
length two. The types distinguish the three ports at each vertex;
they will form the first coordinate of our six colours. The second
coordinate will be a bit assigned to vertices in $A$.

Fix orders on $A$ and $B$. A perfect matching $M$ determines a
bijection $A\to B$, whose sign is denoted by $\sgn(M)$. Changing the
orders multiplies all matching signs by the same constant.

\begin{definition}\label{bip:det-definition}
Let $G$ be a bipartite graph with a perfect matching, and fix linear
orders on its two partition sets. We call $G$ \emph{det-extremal}
if all its perfect matchings have the same sign.
\end{definition}

This property is independent of the chosen orders. For a simple
bipartite graph with biadjacency matrix $B_G$, it is equivalent to
$|\det B_G|=\per B_G$, since the determinant is the signed sum of the
perfect matchings; see~\cite{FJLS2003}.

\begin{lemma}\label{bip:sign-cycle}
Let $G$ be a bipartite graph with a perfect matching $M$, and let $C$
be an $M$-alternating cycle of length $2\ell$. Then
\[
\sgn(M\mathbin{\triangle} E(C))
 =(-1)^{\ell-1}\sgn(M).
\]
Here $\triangle$ denotes symmetric difference.
Moreover, $G$ is not det-extremal if and only if some perfect matching
of $G$ admits an alternating cycle whose length is divisible by four.
\end{lemma}

\begin{proof}
Switching along $C$ induces an $\ell$-cycle on the matched images
of $A\cap V(C)$, so the matching sign is multiplied by
$(-1)^{\ell-1}$. In particular, a cycle of length divisible by four reverses
the sign. Conversely, the symmetric difference of two perfect matchings
of opposite signs is a union of alternating cycles whose sign changes
have negative product. One of these cycles therefore has length
divisible by four.
\end{proof}

To construct the bits, we contract a perfect matching and retain
the direction from $A$ to $B$ on the remaining edges. Alternating
cycles become directed cycles of half the length, so a cycle witnessing
non-det-extremality becomes even. We can alternate the bits around
this cycle and then extend the assignment to the remaining vertices.

\begin{lemma}\label{bip:strong-digraph}
Let $G$ be a connected simple cubic bipartite graph with partition
sets $A$ and $B$, and let $M=\{a_i b_i:i\in I\}$ be a perfect matching,
where $a_i\in A$ and $b_i\in B$. Let $D_M$ be the digraph with vertex
set $I$ and an arc $i\to j$ for every edge $a_i b_j\in E(G)\setminus M$.
Then $D_M$ is strongly connected. Contracting the edges of $M$ induces
a correspondence between $M$-alternating cycles of length $2\ell$ in
$G$ and directed cycles of length $\ell$ in $D_M$.
\end{lemma}

\begin{proof}
Every vertex of $D_M$ has indegree and outdegree two, and its
underlying graph is connected. A source strongly connected component
has no outgoing arcs by degree balance, so it must contain every
vertex.

A directed cycle $i_0i_1\cdots i_{\ell-1}i_0$ expands to the
alternating cycle with edges $a_{i_t}b_{i_{t+1}}$ and
$b_{i_t}a_{i_t}$, with indices modulo $\ell$. Conversely, contracting
the matching edges of an alternating cycle gives a directed cycle,
oriented by its nonmatching edges from $A$ to $B$. The simplicity of
$G$ excludes loops; a pair of opposite arcs corresponds to an
alternating $4$-cycle.
\end{proof}

The condition needed by the colouring formula is that every vertex
in $B$ sees both bits. The equivalence in the next lemma is the
bipartite formulation of \cite[Theorem~3.5]{CKP2026}. Indeed, take
$A$ as the vertex set of a hypergraph and the neighbourhoods $N_G(b)$,
indexed by $b\in B$, as its hyperedges. Its incidence graph is $G$,
and a $2$-colouring of this hypergraph is exactly the required bit
assignment. We retain a direct constructive proof because the
assignment is used in Lemma~\ref{bip:colour-formula}.

\begin{lemma}\label{bip:mixed-bits}
Let $G$ be a connected simple cubic bipartite graph with partition
sets $A$ and $B$. Then $G$ is not det-extremal if and only if there exists
a map $s:A\to\{0,1\}$ such that
\[
 s(N_G(b))=\{0,1\}\qquad\text{for every }b\in B.
\]
\end{lemma}

\begin{proof}
Suppose that $G$ is not det-extremal. By
Lemma~\ref{bip:sign-cycle}, choose a perfect matching $M$ and an
$M$-alternating cycle whose length is divisible by four.
We construct $s$ from these data. By
Lemma~\ref{bip:strong-digraph}, the corresponding cycle in $D_M$
has even length. Assign alternating bits $s(a_i)$ to its vertices $i$.
Whenever an arc $i\to j$ has an assigned tail and an unassigned head,
extend the assignment by setting
$s(a_j)=1-s(a_i)$. Strong connectivity guarantees that this procedure
reaches all vertices. Each vertex $j$ has a predecessor $i$ on the initial
cycle or in this extension with $s(a_i)\ne s(a_j)$. Since both $a_i$ and
$a_j$ are neighbours of $b_j$, every neighbourhood in $B$ contains
both bits.

Conversely, fix a perfect matching $M$. For each $j$, choose an
incoming arc $i\to j$ such that $s(a_i)\ne s(a_j)$; such an arc
exists because $a_j\in N_G(b_j)$ and this neighbourhood contains both
bits. The chosen arcs form a spanning digraph of indegree one, which
contains a directed cycle. The bits alternate on that cycle, so its
length is even. Lemmas~\ref{bip:strong-digraph} and~\ref{bip:sign-cycle}
imply that $G$ is not det-extremal.
\end{proof}

Use the colour set
\[
 \Omega=\{(i,t):i\in\{0,1,2\},\ t\in\{0,1\}\}
\]
and index the ports at each vertex by their types. Thus $c_v(i)$ is
the external colour at the type-$i$ port, $h_v(i)$ is the colour of
the opposite internal edge, and $P_v(i)=\{h_v(j):j\ne i\}$.
At each vertex, complement the bit in each external colour to obtain
the internal palette. To ensure compatibility, place these colours
opposite ports of the same type at vertices in $A$. At a vertex in $B$,
start with this placement and swap the colours opposite the two ports
whose external colours carry the majority bit. This places each colour
shared by adjacent internal palettes opposite their connecting edge
at one or both endpoints.

\begin{lemma}\label{bip:colour-formula}
Let $G$ be a connected simple cubic bipartite graph with partition
sets $A$ and $B$, let $\tau:E(G)\to\{0,1,2\}$ be a proper edge-colouring,
and let $s:A\to\{0,1\}$ satisfy
$s(N_G(b))=\{0,1\}$ for every $b\in B$.
For an edge $ab$ of type $i$, with $a\in A$, set
\begin{equation}\label{bip:external-formula}
 c(ab)=(i,s(a)).
\end{equation}
At $a\in A$, put
\begin{equation}\label{bip:source-formula}
 h_a(i)=(i,1-s(a)).
\end{equation}
At $b\in B$, let $a_i$ be its type-$i$ neighbour and set $t_i=s(a_i)$.
The triple $(t_0,t_1,t_2)$ has a unique minority coordinate, namely
the position of the bit that occurs once. Let $\pi_b$ fix this
coordinate and interchange the other two. Put
\begin{equation}\label{bip:sink-formula}
 h_b(i)=(\pi_b(i),1-t_{\pi_b(i)}).
\end{equation}
Then these assignments define a strong $6$-edge-colouring of $T(G)$.
\end{lemma}

\begin{proof}
At each $a\in A$, the external and internal palettes are
$\{(i,s(a)):i=0,1,2\}$ and $\{(i,1-s(a)):i=0,1,2\}$. At each
$b\in B$, the internal palette is obtained by complementing the bit
in each external colour. Hence every local state uses
the six colours exactly once. It remains to check compatibility
between adjacent triangles. By~\eqref{pre:hidden}, every colour
common to their internal palettes must lie opposite the connecting
edge at one or both endpoints.

For an edge $a_i b$ of type $i$, the intersection of its internal
palettes is
\[
 I_{a_i}\cap I_b=\{(j,1-t_i):t_j=t_i\}.
\]
If $i$ is the minority coordinate, this intersection is
$\{(i,1-t_i)\}$, and its colour is opposite the edge at both ends.
Otherwise, let $j$ be the other majority coordinate. The intersection
is $\{(i,1-t_i),(j,1-t_i)\}$. Its first colour is opposite the edge
at $a_i$, and its second is opposite the edge at $b$, since
$\pi_b(i)=j$. Thus~\eqref{pre:hidden} holds on every edge, and
Lemma~\ref{pre:interface} completes the proof.
\end{proof}

\begin{example}\label{bip:binary-example}
If the neighbour bits at $b$ are $(t_0,t_1,t_2)=(0,0,1)$, then
\[
\begin{array}{c|ccc}
\text{port }i&0&1&2\\ \hline
c_b(i)&(0,0)&(1,0)&(2,1)\\
h_b(i)&(1,1)&(0,1)&(2,0)
\end{array}
\]
Here $\pi_b$ interchanges $0,1$ and fixes $2$. Thus one first
complements the bits of the three external colours and then exchanges
the first two resulting colours; the minority position stays fixed.
\end{example}

The matching used to construct $s$ and the edge-colouring $\tau$ may
be chosen independently. Lemmas~\ref{bip:mixed-bits}
and~\ref{bip:colour-formula} prove Theorem~\ref{bip:main} when $G$ is
not det-extremal.

For multigraphs, define matching signs as before, distinguishing
parallel edges. Then $|\det B_G|=\per B_G$ still characterises
det-extremality when $B_G$ records edge multiplicities.

\begin{corollary}\label{bip:multigraph-nondet}
Let $G$ be a connected loopless cubic bipartite multigraph that is
not det-extremal. Then $\sidx(T(G))=6$.
\end{corollary}

\begin{proof}
Perfect matchings and proper $3$-edge-colourings exist as before.
The sign identity in Lemma~\ref{bip:sign-cycle} also holds for
alternating cycles of length two, since exchanging parallel edges
preserves the induced bijection. The digraph $D_M$ may have loops
and parallel arcs, but each vertex still has indegree and outdegree
two and its underlying graph is connected. The strong-connectivity
argument in Lemma~\ref{bip:strong-digraph} still applies. Loops
correspond to alternating $2$-cycles and cannot lie on a cycle with
alternating bits, so both directions of Lemma~\ref{bip:mixed-bits}
remain valid.

Finally, the three types distinguish incidences even when neighbours
coincide. The formula and interface checks of
Lemma~\ref{bip:colour-formula} are unchanged, and the local lower
bound~\eqref{pre:lowerbound} gives equality.
\end{proof}

The additional difficulties in the det-extremal case are discussed
in Section~\ref{concl:section}.

\subsection{Rooted replacements and Heawood vertex-sums}

For a det-extremal graph, Lemma~\ref{bip:mixed-bits} rules out a
binary assignment that makes every neighbourhood in $B$ mixed.
Removing one vertex from $B$ changes the task: we require mixed
neighbourhoods only at the remaining vertices of $B$, and the three
new boundary signatures must reproduce a prescribed local state.
The following construction achieves both requirements.

For a loopless cubic multigraph $H$ and $v\in V(H)$, write $H-v$ for the
$3$-pole obtained by deleting $v$ and retaining its three incident
half-edges. Its port signatures refer to the retained endpoints, as
in Section~\ref{pre:section}.

\begin{definition}\label{bip:simulator-definition}
Let $\mathcal P$ be a cubic $3$-pole. We call $\mathcal P$ an \emph{exact vertex
simulator} if, for every ordering of its three ports and every choice
of six distinct colours $c_0,c_1,c_2,h_0,h_1,h_2$, its augmented
truncation admits a strong edge-colouring with these six colours whose
port signatures, in the prescribed order, are
\begin{equation}\label{bip:vertex-signatures}
 (c_0,\{h_1,h_2\}),\qquad
 (c_1,\{h_0,h_2\}),\qquad
 (c_2,\{h_0,h_1\}).
\end{equation}
\end{definition}

These are the signatures of a canonical triangle with state
$(c_0,c_1,c_2;h_0,h_1,h_2)$. Thus a simulator can be substituted for
that triangle without changing any of the data seen by the rest of
the coloured graph. Requiring every port order allows any prescribed
attachment, rather than only one fixed correspondence.

\begin{theorem}\label{bip:rooted-simulator}
Let $H$ be a connected simple cubic bipartite det-extremal graph,
and let $v\in V(H)$. Then the $3$-pole $H-v$ is an exact vertex
simulator.
\end{theorem}

\begin{proof}
We use a rooted tree to assign the bits. Its predecessor relations
will give mixed neighbourhoods away from the deleted vertex, and the
absence of even directed cycles will determine the boundary bits.
Choose the bipartition so that $v\in B$, and a perfect matching
$M=\{a_i b_i:i\in I\}$ with $v=b_r$. By
Lemmas~\ref{bip:strong-digraph} and~\ref{bip:sign-cycle}, $D_M$ is
strongly connected and every directed cycle in $D_M$ is odd. Choose
an outgoing spanning arborescence $\mathcal T$ rooted at $r$: a
spanning directed tree with a unique directed path from $r$ to each
vertex. Set $s(a_i)=0$ if the path from $r$ to $i$ in $\mathcal T$
has even length, and $s(a_i)=1$ otherwise.
For every $j\ne r$, its predecessor $i$ in $\mathcal T$ satisfies
$s(a_i)\ne s(a_j)$, and $a_i,a_j$ are both neighbours of $b_j$.
Hence $s(N_H(b))=\{0,1\}$ for all $b\in B\setminus\{v\}$.

For each incoming arc $i\to r$, the tree path from $r$ to $i$
followed by that arc is a directed cycle. Its odd length forces $i$
to have even tree depth. Since $s(a_r)=0$, it follows that
\begin{equation}\label{bip:root-bits}
 s(a)=0\quad\hbox{for all }a\in N_H(v).
\end{equation}

The deleted vertex is therefore the only vertex in $B$ whose
neighbourhood is monochromatic under this assignment. Its removal
leaves exactly the conditions required by the colouring formula.
Choose a proper $3$-edge-colouring $\tau$ of $H$ and apply
\eqref{bip:external-formula}--\eqref{bip:sink-formula} to $H-v$.
The proof of Lemma~\ref{bip:colour-formula} verifies every ordinary
interface. By~\eqref{bip:root-bits}, the signatures of the dangling
ports, indexed by their types, are
\[
 \bigl((i,0),\{(j,1):j\ne i\}\bigr),\qquad i=0,1,2.
\]
Name the three types at $v$ according to the prescribed port order.
The permutation $(i,0)\mapsto c_i$, $(i,1)\mapsto h_i$ gives
\eqref{bip:vertex-signatures}.
\end{proof}

\begin{lemma}\label{bip:vertex-substitution}
Let $K$ be a loopless cubic multigraph, let $\varphi$ be a strong
$6$-edge-colouring of $T(K)$, and let $x\in V(K)$. Let $\mathcal P$ be an
exact vertex simulator, and let $K'$ be obtained by replacing $x$
with $\mathcal P$ according to any bijection between the three incidences at
$x$ and the three ports of $\mathcal P$. Then $T(K')$ admits a strong
$6$-edge-colouring that agrees with $\varphi$ outside the replacement.
\end{lemma}

\begin{proof}
Write the local state at $x$, in attachment order, as
$(c_0,c_1,c_2;h_0,h_1,h_2)$. Realise
\eqref{bip:vertex-signatures} in the simulator. The new signatures
equal the old ones at $x$, so every outside interface remains valid.
The internal interfaces of the pole are valid by definition.
Lemma~\ref{pre:interface} gives the required colouring.
\end{proof}

A \emph{vertex-sum} of disjoint cubic graphs $X,Y$ at vertices $x,y$
is obtained by deleting $x,y$ and joining their neighbour sets by a
bijection. If $X,Y$ are bipartite, choose their bipartitions so that
the new edges join opposite parts. We use the following
characterisation of McCuaig~\cite{McCuaig2000}; see
also~\cite{FJLS2003} and~\cite[Theorem~1.1]{CKP2026}.
\begin{theorem}\label{bip:mccuaig}
Let $G$ be a simple $3$-connected cubic bipartite graph. Then $G$ is
det-extremal if and only if it can be obtained from copies of the
Heawood graph by repeated vertex-sums.
\end{theorem}

McCuaig's characterisation explains why we start from the Heawood
graph. Once its truncation is coloured, Theorem~\ref{bip:rooted-simulator}
supplies the rooted pieces for subsequent vertex-sums, without
requiring colourings of their entire truncations. We first give
the initial Heawood colouring.

Let $\mathbb H$ be the Heawood graph with vertices $A_i,B_i$, indexed by
$\mathbb{Z}_7$, and edges
\[
 A_iB_i,\qquad A_iB_{i+1},\qquad A_iB_{i+3}.
\]
Give these edges types $0,1,2$, respectively. At $A_i$ the port order
is $(B_i,B_{i+1},B_{i+3})$; at $B_i$ it is
$(A_i,A_{i-1},A_{i-3})$. This specifies a proper edge-colouring by types.

\begin{lemma}\label{bip:heawood-seed}
Let $\mathbb H$ be the Heawood graph. Then $T(\mathbb H)$ admits a
strong $6$-edge-colouring.
\end{lemma}

\begin{proof}
Assign the states in Table~\ref{bip:seed-table}, using colours
$0,1,\ldots,5$ and the port orders above. The corresponding colouring
is shown in Figure~\ref{fig:heawood-six-colouring}.
\begin{table}[!p]
\centering
\begin{tabular}{c cc cc}
\toprule
$i$ & $c_{A_i}$ & $h_{A_i}$ & $c_{B_i}$ & $h_{B_i}$\\
\midrule
0 & $(0,2,4)$ & $(1,3,5)$ & $(0,3,5)$ & $(1,2,4)$\\
1 & $(1,2,4)$ & $(0,3,5)$ & $(1,2,5)$ & $(3,0,4)$\\
2 & $(0,3,4)$ & $(5,2,1)$ & $(0,2,4)$ & $(1,5,3)$\\
3 & $(1,3,5)$ & $(2,0,4)$ & $(1,3,4)$ & $(0,5,2)$\\
4 & $(1,2,5)$ & $(0,3,4)$ & $(1,3,4)$ & $(5,2,0)$\\
5 & $(0,2,5)$ & $(1,4,3)$ & $(0,2,4)$ & $(3,1,5)$\\
6 & $(1,3,4)$ & $(0,2,5)$ & $(1,2,5)$ & $(4,3,0)$\\
\bottomrule
\end{tabular}
\caption{A complete $6$-colour state assignment for the Heawood graph.}
\label{bip:seed-table}
\end{table}

Every state contains the six distinct colours, and the external
entries agree at the ends of every edge. To verify
\eqref{pre:hidden}, put $d_0=0,d_1=1,d_2=3$ and define
\[
 J_{i,k}=I_{A_i}\cap I_{B_{i+d_k}},\qquad
 K_{i,k}=\{h_{A_i}(k),h_{B_{i+d_k}}(k)\}.
\]
Table~\ref{bip:seed-interfaces} lists $J_{i,k};K_{i,k}$ for all
$21$ edges. Each first set is contained in the second, as required.
\begin{table}[!p]
\centering
\begin{tabular}{c ccc}
\toprule
$i$ & type $0$ & type $1$ & type $2$\\
\midrule
0 & $\{1\};\{1\}$ & $\{3\};\{0,3\}$ & $\{5\};\{2,5\}$\\
1 & $\{0,3\};\{0,3\}$ & $\{3,5\};\{3,5\}$ & $\{0,5\};\{0,5\}$\\
2 & $\{1,5\};\{1,5\}$ & $\{2,5\};\{2,5\}$ & $\{1,5\};\{1,5\}$\\
3 & $\{0,2\};\{0,2\}$ & $\{0,2\};\{0,2\}$ & $\{0,4\};\{0,4\}$\\
4 & $\{0\};\{0,5\}$ & $\{3\};\{1,3\}$ & $\{4\};\{4\}$\\
5 & $\{1,3\};\{1,3\}$ & $\{3,4\};\{3,4\}$ & $\{3,4\};\{3,4\}$\\
6 & $\{0\};\{0,4\}$ & $\{2\};\{2\}$ & $\{5\};\{3,5\}$\\
\bottomrule
\end{tabular}
\caption{Verification of every interface in the Heawood seed.}
\label{bip:seed-interfaces}
\end{table}

The result follows from Lemma~\ref{pre:interface}.
\end{proof}

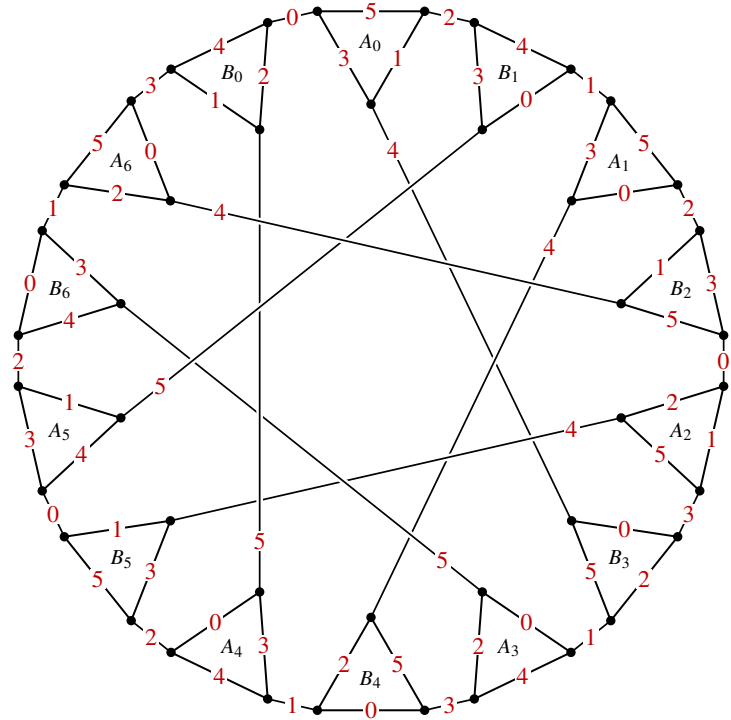
\begin{figure}[!p]
\centering
\begin{tikzpicture}[x=.78cm,y=.78cm,line cap=round,line join=round]
\tikzset{hc/.style={text=red!75!black,fill=white,inner sep=.65pt,font=\footnotesize},he/.style={draw=black,line width=.65pt,preaction={draw=white,line width=2.5pt}},hi/.style={draw=black,line width=.75pt},hp/.style={circle,fill=black,inner sep=1.25pt}}
\coordinate (A0p0) at (-0.90933,5.92500);
\coordinate (A0p1) at (0.90933,5.92500);
\coordinate (A0p2) at (0.00000,4.35000);
\coordinate (B1p0) at (3.39003,4.94370);
\coordinate (B1p1) at (1.75149,5.73279);
\coordinate (B1p2) at (1.88739,3.91921);
\coordinate (A1p0) at (4.06539,4.40512);
\coordinate (A1p1) at (5.19930,2.98324);
\coordinate (A1p2) at (3.40097,2.71219);
\coordinate (B2p0) at (5.97879,0.43190);
\coordinate (B2p1) at (5.57409,2.20496);
\coordinate (B2p2) at (4.24094,0.96797);
\coordinate (A2p0) at (5.97879,-0.43190);
\coordinate (A2p1) at (5.57409,-2.20496);
\coordinate (A2p2) at (4.24094,-0.96797);
\coordinate (B3p0) at (4.06539,-4.40512);
\coordinate (B3p1) at (5.19930,-2.98324);
\coordinate (B3p2) at (3.40097,-2.71219);
\coordinate (A3p0) at (3.39003,-4.94370);
\coordinate (A3p1) at (1.75149,-5.73279);
\coordinate (A3p2) at (1.88739,-3.91921);
\coordinate (B4p0) at (-0.90933,-5.92500);
\coordinate (B4p1) at (0.90933,-5.92500);
\coordinate (B4p2) at (0.00000,-4.35000);
\coordinate (A4p0) at (-1.75149,-5.73279);
\coordinate (A4p1) at (-3.39003,-4.94370);
\coordinate (A4p2) at (-1.88739,-3.91921);
\coordinate (B5p0) at (-5.19930,-2.98324);
\coordinate (B5p1) at (-4.06539,-4.40512);
\coordinate (B5p2) at (-3.40097,-2.71219);
\coordinate (A5p0) at (-5.57409,-2.20496);
\coordinate (A5p1) at (-5.97879,-0.43190);
\coordinate (A5p2) at (-4.24094,-0.96797);
\coordinate (B6p0) at (-5.57409,2.20496);
\coordinate (B6p1) at (-5.97879,0.43190);
\coordinate (B6p2) at (-4.24094,0.96797);
\coordinate (A6p0) at (-5.19930,2.98324);
\coordinate (A6p1) at (-4.06539,4.40512);
\coordinate (A6p2) at (-3.40097,2.71219);
\coordinate (B0p0) at (-1.75149,5.73279);
\coordinate (B0p1) at (-3.39003,4.94370);
\coordinate (B0p2) at (-1.88739,3.91921);
\draw[he] (A0p2) -- (B3p2);
\draw[he] (A1p2) -- (B4p2);
\draw[he] (A2p2) -- (B5p2);
\draw[he] (A3p2) -- (B6p2);
\draw[he] (A4p2) -- (B0p2);
\draw[he] (A5p2) -- (B1p2);
\draw[he] (A6p2) -- (B2p2);
\draw[he] (A0p0) -- (B0p0);
\draw[he] (A0p1) -- (B1p1);
\draw[he] (A1p0) -- (B1p0);
\draw[he] (A1p1) -- (B2p1);
\draw[he] (A2p0) -- (B2p0);
\draw[he] (A2p1) -- (B3p1);
\draw[he] (A3p0) -- (B3p0);
\draw[he] (A3p1) -- (B4p1);
\draw[he] (A4p0) -- (B4p0);
\draw[he] (A4p1) -- (B5p1);
\draw[he] (A5p0) -- (B5p0);
\draw[he] (A5p1) -- (B6p1);
\draw[he] (A6p0) -- (B6p0);
\draw[he] (A6p1) -- (B0p1);
\draw[hi] (A0p1) -- node[hc] {1} (A0p2);
\draw[hi] (A0p0) -- node[hc] {3} (A0p2);
\draw[hi] (A0p0) -- node[hc] {5} (A0p1);
\node[font=\scriptsize,inner sep=0pt] at (0.00000,5.40000) {$A_{0}$};
\node[hp] at (A0p0) {};
\node[hp] at (A0p1) {};
\node[hp] at (A0p2) {};
\draw[hi] (B1p1) -- node[hc] {3} (B1p2);
\draw[hi] (B1p0) -- node[hc] {0} (B1p2);
\draw[hi] (B1p0) -- node[hc] {4} (B1p1);
\node[font=\scriptsize,inner sep=0pt] at (2.34297,4.86523) {$B_{1}$};
\node[hp] at (B1p0) {};
\node[hp] at (B1p1) {};
\node[hp] at (B1p2) {};
\draw[hi] (A1p1) -- node[hc] {0} (A1p2);
\draw[hi] (A1p0) -- node[hc] {3} (A1p2);
\draw[hi] (A1p0) -- node[hc] {5} (A1p1);
\node[font=\scriptsize,inner sep=0pt] at (4.22189,3.36685) {$A_{1}$};
\node[hp] at (A1p0) {};
\node[hp] at (A1p1) {};
\node[hp] at (A1p2) {};
\draw[hi] (B2p1) -- node[hc] {1} (B2p2);
\draw[hi] (B2p0) -- node[hc] {5} (B2p2);
\draw[hi] (B2p0) -- node[hc] {3} (B2p1);
\node[font=\scriptsize,inner sep=0pt] at (5.26461,1.20161) {$B_{2}$};
\node[hp] at (B2p0) {};
\node[hp] at (B2p1) {};
\node[hp] at (B2p2) {};
\draw[hi] (A2p1) -- node[hc] {5} (A2p2);
\draw[hi] (A2p0) -- node[hc] {2} (A2p2);
\draw[hi] (A2p0) -- node[hc] {1} (A2p1);
\node[font=\scriptsize,inner sep=0pt] at (5.26461,-1.20161) {$A_{2}$};
\node[hp] at (A2p0) {};
\node[hp] at (A2p1) {};
\node[hp] at (A2p2) {};
\draw[hi] (B3p1) -- node[hc] {0} (B3p2);
\draw[hi] (B3p0) -- node[hc] {5} (B3p2);
\draw[hi] (B3p0) -- node[hc] {2} (B3p1);
\node[font=\scriptsize,inner sep=0pt] at (4.22189,-3.36685) {$B_{3}$};
\node[hp] at (B3p0) {};
\node[hp] at (B3p1) {};
\node[hp] at (B3p2) {};
\draw[hi] (A3p1) -- node[hc] {2} (A3p2);
\draw[hi] (A3p0) -- node[hc] {0} (A3p2);
\draw[hi] (A3p0) -- node[hc] {4} (A3p1);
\node[font=\scriptsize,inner sep=0pt] at (2.34297,-4.86523) {$A_{3}$};
\node[hp] at (A3p0) {};
\node[hp] at (A3p1) {};
\node[hp] at (A3p2) {};
\draw[hi] (B4p1) -- node[hc] {5} (B4p2);
\draw[hi] (B4p0) -- node[hc] {2} (B4p2);
\draw[hi] (B4p0) -- node[hc] {0} (B4p1);
\node[font=\scriptsize,inner sep=0pt] at (0.00000,-5.40000) {$B_{4}$};
\node[hp] at (B4p0) {};
\node[hp] at (B4p1) {};
\node[hp] at (B4p2) {};
\draw[hi] (A4p1) -- node[hc] {0} (A4p2);
\draw[hi] (A4p0) -- node[hc] {3} (A4p2);
\draw[hi] (A4p0) -- node[hc] {4} (A4p1);
\node[font=\scriptsize,inner sep=0pt] at (-2.34297,-4.86523) {$A_{4}$};
\node[hp] at (A4p0) {};
\node[hp] at (A4p1) {};
\node[hp] at (A4p2) {};
\draw[hi] (B5p1) -- node[hc] {3} (B5p2);
\draw[hi] (B5p0) -- node[hc] {1} (B5p2);
\draw[hi] (B5p0) -- node[hc] {5} (B5p1);
\node[font=\scriptsize,inner sep=0pt] at (-4.22189,-3.36685) {$B_{5}$};
\node[hp] at (B5p0) {};
\node[hp] at (B5p1) {};
\node[hp] at (B5p2) {};
\draw[hi] (A5p1) -- node[hc] {1} (A5p2);
\draw[hi] (A5p0) -- node[hc] {4} (A5p2);
\draw[hi] (A5p0) -- node[hc] {3} (A5p1);
\node[font=\scriptsize,inner sep=0pt] at (-5.26461,-1.20161) {$A_{5}$};
\node[hp] at (A5p0) {};
\node[hp] at (A5p1) {};
\node[hp] at (A5p2) {};
\draw[hi] (B6p1) -- node[hc] {4} (B6p2);
\draw[hi] (B6p0) -- node[hc] {3} (B6p2);
\draw[hi] (B6p0) -- node[hc] {0} (B6p1);
\node[font=\scriptsize,inner sep=0pt] at (-5.26461,1.20161) {$B_{6}$};
\node[hp] at (B6p0) {};
\node[hp] at (B6p1) {};
\node[hp] at (B6p2) {};
\draw[hi] (A6p1) -- node[hc] {0} (A6p2);
\draw[hi] (A6p0) -- node[hc] {2} (A6p2);
\draw[hi] (A6p0) -- node[hc] {5} (A6p1);
\node[font=\scriptsize,inner sep=0pt] at (-4.22189,3.36685) {$A_{6}$};
\node[hp] at (A6p0) {};
\node[hp] at (A6p1) {};
\node[hp] at (A6p2) {};
\draw[hi] (B0p1) -- node[hc] {1} (B0p2);
\draw[hi] (B0p0) -- node[hc] {2} (B0p2);
\draw[hi] (B0p0) -- node[hc] {4} (B0p1);
\node[font=\scriptsize,inner sep=0pt] at (-2.34297,4.86523) {$B_{0}$};
\node[hp] at (B0p0) {};
\node[hp] at (B0p1) {};
\node[hp] at (B0p2) {};
\node[hc] at ($(A0p0)!0.5!(B0p0)$) {0};
\node[hc] at ($(A0p1)!0.5!(B1p1)$) {2};
\node[hc] at ($(A0p2)!0.11!(B3p2)$) {4};
\node[hc] at ($(A1p0)!0.5!(B1p0)$) {1};
\node[hc] at ($(A1p1)!0.5!(B2p1)$) {2};
\node[hc] at ($(A1p2)!0.11!(B4p2)$) {4};
\node[hc] at ($(A2p0)!0.5!(B2p0)$) {0};
\node[hc] at ($(A2p1)!0.5!(B3p1)$) {3};
\node[hc] at ($(A2p2)!0.11!(B5p2)$) {4};
\node[hc] at ($(A3p0)!0.5!(B3p0)$) {1};
\node[hc] at ($(A3p1)!0.5!(B4p1)$) {3};
\node[hc] at ($(A3p2)!0.11!(B6p2)$) {5};
\node[hc] at ($(A4p0)!0.5!(B4p0)$) {1};
\node[hc] at ($(A4p1)!0.5!(B5p1)$) {2};
\node[hc] at ($(A4p2)!0.11!(B0p2)$) {5};
\node[hc] at ($(A5p0)!0.5!(B5p0)$) {0};
\node[hc] at ($(A5p1)!0.5!(B6p1)$) {2};
\node[hc] at ($(A5p2)!0.11!(B1p2)$) {5};
\node[hc] at ($(A6p0)!0.5!(B6p0)$) {1};
\node[hc] at ($(A6p1)!0.5!(B0p1)$) {3};
\node[hc] at ($(A6p2)!0.11!(B2p2)$) {4};
\end{tikzpicture}
\caption{A strong $6$-edge-colouring of $T(\mathbb H)$ in the classical circular layout of the Heawood graph. Clockwise from the top, the canonical triangles follow the Hamiltonian cycle $A_0,B_1,A_1,B_2,\ldots,A_6,B_0$; the seven chords correspond to the edges $A_iB_{i+3}$. Red numerals are the edge colours $0,\ldots,5$ from Table~\ref{bip:seed-table}.}
\label{fig:heawood-six-colouring}
\end{figure}

To follow the vertex-sum expression supplied by
Theorem~\ref{bip:mccuaig} recursively, we need the intermediate
graphs to be det-extremal so that
Theorem~\ref{bip:rooted-simulator} applies. McCuaig's characterisation
gives this once their $3$-connectivity is established, as in the next
lemma.

\begin{lemma}\label{bip:vertex-sum-connectivity}
Let $X$ and $Y$ be disjoint simple $3$-connected cubic bipartite
graphs, and let $x\in V(X)$ and $y\in V(Y)$. Every vertex-sum of
$X$ and $Y$ at $x$ and $y$ is a simple $3$-connected cubic bipartite
graph.
\end{lemma}

\begin{proof}
Simplicity, cubicity and bipartiteness are immediate. Let $S$ contain
at most two vertices of the vertex-sum. If $S$ meets each side in at
most one vertex, both sides remain connected by the $3$-connectivity
of $X,Y$, and at least one joining edge survives. Otherwise, by
symmetry, $S\subseteq V(X)\setminus\{x\}$. The side $Y-y$ remains
connected. Every component of $X-(S\cup\{x\})$ contains a surviving
neighbour of $x$, since otherwise it would be separated from $x$ in
$X-S$. Each such component therefore has a joining edge to $Y-y$.
Thus deletion of $S$ leaves a connected graph.
\end{proof}

We can now combine the Heawood seed with the replacement construction.
At each vertex-sum, a colouring on one side supplies the boundary
state to be realised by the simulator on the other side.

\begin{lemma}\label{bip:three-connected-det}
Let $G$ be a simple $3$-connected cubic bipartite det-extremal graph.
Then $T(G)$ admits a strong $6$-edge-colouring, which can be
constructed from a vertex-sum expression for $G$ in terms of Heawood
graphs.
\end{lemma}

\begin{proof}
By Theorem~\ref{bip:mccuaig}, $G$ has an expression as repeated
vertex-sums of Heawood graphs. Every intermediate graph is simple and
$3$-connected by Lemma~\ref{bip:vertex-sum-connectivity}, and
det-extremal by Theorem~\ref{bip:mccuaig}. We induct on this
expression. Lemma~\ref{bip:heawood-seed} colours each initial copy.
If an intermediate graph is a vertex-sum of $X,Y$ at $x,y$, colour
$T(X)$ by induction and replace $x$ by $Y-y$. We need only the
colouring on the $X$ side: the pole $Y-y$ is assigned the signatures
of the removed triangle at $x$.
Theorem~\ref{bip:rooted-simulator} and
Lemma~\ref{bip:vertex-substitution} extend the colouring across this
replacement.
\end{proof}

\subsection{2-edge cuts and completion of the proof}

It remains to colour the det-extremal graphs of vertex connectivity
two. We shall choose a $2$-edge cut whose closures are simple, with
one closure $3$-connected and det-extremal and the other non-det-extremal.
The Heawood construction colours the truncation of the first closure,
and the binary formula colours that of the second. We first show how
to combine these colourings, then obtain the required cut from the
structural decomposition and an alternating-cycle argument.

\begin{lemma}\label{bip:edge-sum}
Let $P$ and $R$ be disjoint loopless cubic multigraphs, and suppose
that both $T(P)$ and $T(R)$ admit strong $6$-edge-colourings. Choose
edges $e=uv\in E(P)$ and $f=xy\in E(R)$, and let $G$ be obtained
from $P\sqcup R$ by deleting $e,f$ and adding $uy,xv$. Then $T(G)$
admits a strong $6$-edge-colouring.
\end{lemma}

\begin{proof}
Let $c$ be the external colour at $e$, and let $U,V$ be its
exposed pairs at $u,v$. These are disjoint $2$-element sets avoiding
$c$. Hence $\{c\},U,V$ and the remaining singleton partition the
six colours. Denote the corresponding data at $f$ by $d,X,Y$,
where $X$ is exposed at $x$ and $Y$ at $y$.

Permute the colours of $T(R)$ so that $d$ maps to $c$, $X$ to $U$,
$Y$ to $V$, and the remaining colour to the remaining colour.
Assign $c$ to both new external edges. On $uy$ the exposed pairs
are $U,V$, and on $xv$ they are again $U,V$. Hence both interfaces
are compatible, and all local states are preserved. Apply
Lemma~\ref{pre:interface}.
\end{proof}

For a $2$-edge cut $\delta(S)$, its \emph{shores} are $S$ and
$V(G)\setminus S$. When the two boundary endpoints on a shore are
distinct, adding an edge between them to the subgraph induced by that
shore gives its \emph{closure}; the added edge is its \emph{marker}.
The next lemma shows that both closures are connected cubic bipartite
multigraphs. A closure may initially have parallel edges; simplicity
will be established when the appropriate cut is chosen.

\begin{lemma}\label{bip:cuts}
Let $G$ be a connected simple cubic bipartite graph with partition
sets $A$ and $B$. Then the following statements hold.
\begin{enumerate}[label=(\roman*)]
\item $G$ is bridgeless and has no cut vertex.
\item $G$ has vertex connectivity two if and only if it has a
 $2$-edge cut.
\item For every $2$-edge cut $\delta_G(S)$, each shore induces a
 connected subgraph and contains equally many vertices of $A$ and
 $B$. The two boundary endpoints on each shore are distinct, with
 one in each partition set. Adding an edge between them produces
 a connected cubic bipartite multigraph, and every pair of parallel
 edges in this multigraph contains the added edge.
\end{enumerate}
\end{lemma}
\begin{proof}
For the bipartition $A,B$ and a vertex set $S$, put
$d_A(S)=|E(A\cap S,B\setminus S)|$ and
$d_B(S)=|E(B\cap S,A\setminus S)|$. Cubicity gives
\begin{equation}\label{bip:cut-balance}
 3\bigl(|A\cap S|-|B\cap S|\bigr)=d_A(S)-d_B(S).
\end{equation}
We first prove~(i). A $1$-edge cut would make the right-hand side
equal to $1$ or $-1$, so $G$ is bridgeless. A cut vertex would separate
at least two components, each requiring at least two edges to that
vertex, which is impossible at degree three. Thus $G$ has connectivity
two or three.

We next establish~(iii), which will also be used to prove~(ii).
If $|\delta(S)|=2$, equation~\eqref{bip:cut-balance} implies
$d_A(S)=d_B(S)=1$ and $|A\cap S|=|B\cap S|$. Thus the two terminals
on each side of the cut are distinct and lie in opposite partition
sets. Both shores are connected: if one had two components, each
would need at least two boundary edges because there are no bridges.
Adding an edge between the terminals on either shore produces a
connected cubic bipartite multigraph. Since the original graph is
simple, the only possible parallelism in either closure involves its
marker.

It remains to prove~(ii). Suppose that deleting $x,y$ disconnects $G$.
Every resulting component has at least two boundary edges, since
there are no bridges. If one has exactly two, the desired edge cut
has already been found. Otherwise there are exactly two components,
say $C,D$, each with three boundary edges, and $xy$ is not an edge.
Since neither $x$ nor $y$ is a cut vertex, both have neighbours in
both components. After interchanging $x,y$ if necessary, $C$ has one
edge to $x$ and two to $y$, while $D$ has two to $x$ and one to $y$.
Then $C\cup\{y\}$ has a $2$-edge boundary. Conversely, the two
terminals on one shore of a $2$-edge cut form a vertex cut: they are
distinct. That shore contains another vertex, since a $2$-vertex shore
would have at most one internal edge and hence at least four boundary
edges in a simple cubic graph. This proves~(ii), and also shows that a $3$-connected cubic graph
is $3$-edge-connected.

\end{proof}

An arbitrary $2$-edge cut need not give a $3$-connected closure.
We first recall the structural result that permits a suitable choice.
For a simple $2$-connected cubic bipartite graph $G$, declare two
vertices equivalent if they coincide or are joined by three
edge-disjoint paths.
Contracting the equivalence classes and deleting loops gives a cactus
$C$: a connected multigraph whose blocks are cycles, including
cycles of length two.

At a vertex $x$ of $C$, start with the subgraph induced by its class.
For each pair of edges incident with $x$ that forms a $2$-edge cut of
$C$, add a marker between their endpoints in that class. The resulting
\emph{piece} $G_x$ is either a $3$-bond (two vertices joined by three
parallel edges) or a simple $3$-connected cubic bipartite graph.
These are the decomposition and pieces of
\cite[Section~2]{FJLS2003}; see also
\cite[Definition~5.1 and Observation~5.2]{CKP2026}.
We use the following degree-$2$ specialisation of the structural theorem.

\begin{theorem}
\label{bip:end-piece-structure}
Let $G$ be a simple cubic bipartite det-extremal graph of connectivity
two, with decomposition cactus $C$ and pieces $G_x$ as above.
If $\deg_C(x)=2$, then $G_x$ is det-extremal.
\end{theorem}

This is the degree-$2$ case of the structural theorem of
Funk, Jackson, Labbate and Sheehan~\cite{FJLS2003}, stated in
\cite[Theorem~5.4]{CKP2026}. The next lemma gives both closures
needed for the colouring argument.

\begin{lemma}\label{bip:end-piece}
Let $G$ be a connected simple cubic bipartite det-extremal graph
with vertex connectivity two. Then $G$ has a $2$-edge cut with simple
closures $P,R$ such that $P$ is $3$-connected and det-extremal,
whereas $R$ is not det-extremal.
\end{lemma}

\begin{proof}
The decomposition cactus $C$ has at least two vertices, since $G$
has connectivity two. Choose an end cycle of $C$, that is, a cycle
containing at most one cut vertex, and let $z$ be a non-cut vertex on
it. Then $\deg_C(z)=2$.
Its class $S$ therefore has exactly two boundary edges, and the
piece $P=G_z$ is the closure of $G[S]$ with exactly one marker.
Theorem~\ref{bip:end-piece-structure} makes $P$ det-extremal.
If $P$ were a $3$-bond, its two nonmarker edges would be parallel
edges of $G$, contradicting simplicity. The description of the
pieces therefore makes $P$ simple and $3$-connected.

Let $R$ be the other closure. By Lemma~\ref{bip:cuts}, both closures
are connected cubic bipartite multigraphs. Write their markers as
$e=uv$ in $P$ and $f=xy$ in $R$, so that $G$ is obtained by deleting
$e,f$ and adding $uy,xv$.
Choose perfect matchings $p_0,p_1$ of $P$ avoiding and containing
$e$, respectively, and $r_0,r_1$ of $R$ avoiding and containing $f$.
Such matchings exist because each closure admits a proper
$3$-edge-colouring. Let $C_P$ and $C_R$ be the cycles containing
the markers in $p_0\mathbin{\triangle}p_1$ and
$r_0\mathbin{\triangle}r_1$, respectively. Cycles of length two are
allowed in $R$. Since $P$ is det-extremal,
Lemma~\ref{bip:sign-cycle} gives $|C_P|\equiv2\pmod4$.

Joining the paths $C_P-e$ and $C_R-f$ with $uy,xv$ gives a cycle
in $G$ of length $|C_P|+|C_R|$. It is alternating with respect to
the perfect matching
\[
 M=\bigl((p_1\cup r_1)\setminus\{e,f\}\bigr)\cup\{uy,xv\}.
\]
Indeed, the two joining edges belong to $M$, and each path alternates
with its first and last edges outside $M$. Since $G$ is det-extremal,
another application of Lemma~\ref{bip:sign-cycle} yields
\[
 |C_P|+|C_R|\equiv2\pmod4,
 \qquad\text{hence}\qquad |C_R|\equiv0\pmod4.
\]

This holds for every choice of $r_0,r_1$ as above. If $f$ were parallel
to an old edge $g$, choose $r_1$ containing $f$ and put
$r_0=(r_1\setminus\{f\})\cup\{g\}$. Then $C_R$ would have length
two, a contradiction. By Lemma~\ref{bip:cuts}, every possible pair
of parallel edges in $R$ contains its marker, so $R$ is simple.
Finally, $C_R$ is an $r_1$-alternating cycle of length divisible by
four. Lemma~\ref{bip:sign-cycle} therefore shows that $R$ is not
det-extremal.
\end{proof}

\begin{proof}[Proof of Theorem~\ref{bip:main}]
If $G$ is not det-extremal, Lemmas~\ref{bip:mixed-bits}
and~\ref{bip:colour-formula} give a strong $6$-edge-colouring of
$T(G)$. We may therefore suppose that $G$ is det-extremal. By
Lemma~\ref{bip:cuts}, its connectivity is two or three. In the
$3$-connected case, Lemma~\ref{bip:three-connected-det} supplies
the required colouring through a Heawood vertex-sum expression.

If the connectivity is two, take the closures $P,R$ supplied by
Lemma~\ref{bip:end-piece}. Lemma~\ref{bip:three-connected-det}
colours $T(P)$, while Lemmas~\ref{bip:mixed-bits}
and~\ref{bip:colour-formula} colour $T(R)$.
Lemma~\ref{bip:edge-sum} combines these colourings along the marker
edges to give a strong $6$-edge-colouring of $T(G)$. The lower bound
follows from~\eqref{pre:lowerbound}.
\end{proof}

\FloatBarrier
\section{Claw-free cubic graphs}\label{cf:section}

We now prove Theorem~\ref{cf:main}, using the palette
$\Omega=\{1,2,3,4,5,6\}$. Throughout the proof, $G$ is the original
simple bridgeless cubic claw-free base, and $T(G)$ is the graph to be
coloured.

When $G\ne K_4$ and a triangle block is present, suppressing the
diamond chains produces a reduced base $G_0$. Contracting the retained
triangle blocks gives a loopless cubic multigraph $Q$, called the
skeleton, and $G_0\cong T(Q)$.
We first justify this reduction, then construct a colouring of
$T(G_0)$ by working with the triangle blocks of $G_0$. Finally, we restore
the diamond chains to obtain the required colouring of $T(G)$.
The result for double truncations will follow as a corollary.
The cases $G=K_4$ and graphs consisting only of diamond blocks are
handled separately by the same diamond colouring.

\subsection{The block decomposition and the reduced base}

An induced diamond has tips $a,b$ and centres $u,v$, with edges
\[
 au,av,bu,bv,uv.
\]
The tips are nonadjacent. In a cubic supergraph each tip has precisely
one edge leaving the diamond, whereas its centres have no outside
neighbours.

\begin{lemma}\label{cf:blocks}
Let $G$ be a connected simple cubic claw-free graph with $G\ne K_4$.
Then the vertex sets of the induced diamonds of $G$, together with the
vertex sets of the triangles not contained in an induced diamond, form
a partition of $V(G)$. Each diamond block has two edges leaving it,
and each triangle block has three.
\end{lemma}

\begin{proof}
By claw-freeness, every vertex belongs to a triangle.
If two distinct triangles meet at a vertex, cubicity forces them to
share an edge. The four vertices in their union induce either a diamond
or $K_4$. In the latter case all four vertices are saturated, so
connectedness gives $G=K_4$, contrary to the hypothesis.

Every triangle meeting an induced diamond $D$ is contained in $D$:
the centres of $D$ already have degree three within $D$, and the outside
neighbour of either tip cannot be adjacent to a centre.

Now suppose two induced diamonds meet. Of the two triangles forming
the second diamond, one contains a common vertex and therefore lies
inside the first diamond by the preceding observation. The other
triangle shares an edge with the first triangle, so it also meets
the first diamond and lies inside it. Thus the two diamonds have the same four vertices
and are equal.

Finally, a vertex outside all diamonds belongs to a triangle. This
triangle cannot meet a diamond, by the claim. It cannot meet another
triangle either, since their union would be a diamond or the excluded
$K_4$. Hence the triangles on the remaining vertices are pairwise
disjoint and unique. Cubicity gives the stated numbers of edges leaving
the blocks.
\end{proof}

Contracting each block of Lemma~\ref{cf:blocks} gives a connected
loopless multigraph $J$. A vertex representing a triangle block has
degree three, and one representing a diamond block has degree two;
we call these triangle vertices and diamond vertices, respectively.
Edges between distinct blocks and their attachment incidences are
retained individually.
This is the triangle-and-diamond
decomposition underlying \cite[Proposition~1]{Oum2011}.
The next lemma describes the chains formed by diamond vertices and
shows that, when triangle vertices are present, bridgelessness forces
each chain to have distinct end blocks, so that suppression creates
no loop.

\Needspace{11\baselineskip}
\begin{lemma}\label{cf:skeleton}
Let $G$ be a connected simple bridgeless cubic claw-free graph with
$G\ne K_4$, and let $J$ be the multigraph obtained by contracting
each block of the partition in Lemma~\ref{cf:blocks}. Then the
following statements hold.
\begin{enumerate}[label=(\roman*)]
\item If every block is a diamond, then $J$ is a cycle, where two
vertices joined by two parallel edges are regarded as a cycle of
length two.
\item If at least one block is a triangle, then $E(J)$ has a partition
into paths with distinct triangle vertices as endpoints and diamond
vertices as internal vertices. Paths with no internal vertices are
allowed. Replacing each such path by a distinct edge yields a connected
loopless cubic multigraph $Q$.
\end{enumerate}
In case (ii), the multigraph $Q$ is called the \emph{skeleton} of $G$.
Each edge of $Q$ records its corresponding diamond chain and the two
attachment ports in $G$.
\end{lemma}

\begin{proof}
Since the blocks are connected and $G$ is bridgeless, their contraction
$J$ is also bridgeless.

If all vertices of $J$ have degree two, $J$ is a cycle, possibly of
length two. Otherwise, since $J$ is connected, following the unique
continuation through each degree-$2$ diamond vertex partitions $E(J)$
into trails between triangle vertices, possibly with coincident endpoints.

Such a trail cannot return to its starting triangle vertex $t$. If it did,
the set consisting of $t$ and its internal diamond vertices would have
exactly one edge leaving it, namely the unused third incidence at $t$.
This would be a bridge of $J$. Thus these trails are paths with distinct
triangle endpoints. Their replacement by edges creates no loops;
parallel edges may occur. Each triangle supplies exactly three
incidences, so $Q$ is cubic, and connectedness follows from that of $J$.
\end{proof}

Each path in Lemma~\ref{cf:skeleton}(ii) represents a \emph{diamond
chain} in $G$: a sequence of diamond blocks joined between two triangle
blocks. The length of the chain is its number of diamonds, so a direct
edge between triangle blocks is a chain of length zero. Replacing each
non-empty chain by an edge joining its original attachment vertices,
while retaining the direct edges between triangle blocks, defines the
reduced base $G_0$; see Figure~\ref{fig:diamond-reduction}.

\begin{lemma}\label{cf:reduced-base}
Let $G$ be a connected simple bridgeless cubic claw-free graph with
$G\ne K_4$, and suppose that its block partition contains a triangle.
Let $G_0$ be its reduced base and $Q$ its skeleton.
Then $G_0$ is a connected simple cubic diamond-free
graph, and $G_0\cong T(Q)$.
Moreover, $G$ is obtained from $G_0$ by replacing the corresponding edges
by their recorded diamond chains.
\end{lemma}
\begin{proof}
By Lemma~\ref{cf:skeleton}, the skeleton $Q$ is connected, loopless
and cubic. At each retained triangle block, the three attachment
vertices correspond bijectively to the three incidences at the
associated vertex of $Q$. Under this correspondence, the block edges
are the canonical triangle edges, and each joining edge represents
the corresponding edge of $Q$. This is the incidence definition of
$T(Q)$, so $G_0\cong T(Q)$, including when $Q$ has parallel edges.

The properties of truncations recalled in Section~\ref{intro:section}
therefore make $G_0$ simple, cubic and diamond-free; it is connected
because $Q$ is connected. The recorded attachment vertices and chain
lengths determine the restoration of $G$.
\end{proof}

\begin{figure}[!htbp]
\centering
\begin{tikzpicture}[x=1.3cm,y=1cm,line width=.65pt,
 vertex/.style={circle,fill=black,inner sep=1.4pt},
 every node/.style={font=\small}]
\begin{scope}[yshift=1.35cm]
 \coordinate (x) at (.55,0);
 \coordinate (l1) at (0,.45); \coordinate (l2) at (0,-.45);
 \coordinate (y) at (7.45,0);
 \coordinate (r1) at (8,.45); \coordinate (r2) at (8,-.45);
 \draw (x)--(l1)--(l2)--cycle;
 \draw (y)--(r1)--(r2)--cycle;
 \draw (l1)--(-.4,.75) (l2)--(-.4,-.75)
       (r1)--(8.4,.75) (r2)--(8.4,-.75);
 \foreach \j/\t in {1/2,2/6} {
  \coordinate (a\j) at (\t-.6,0);
  \coordinate (b\j) at (\t+.6,0);
  \coordinate (u\j) at (\t,.48);
  \coordinate (v\j) at (\t,-.48);
  \draw (a\j)--(u\j)--(b\j)--(v\j)--cycle (u\j)--(v\j);
  \foreach \p in {a\j,b\j,u\j,v\j} \node[vertex] at (\p) {};
 }
 \draw (x)--(a1) (b1)--(3.45,0) (4.55,0)--(a2) (b2)--(y);
 \node at (4,0) {$\cdots$};
 \foreach \p in {x,l1,l2,y,r1,r2} \node[vertex] at (\p) {};
 \node[below=3pt] at (x) {$x$}; \node[below=3pt] at (y) {$y$};
 \node at (2,.88) {$D_1$}; \node at (6,.88) {$D_k$};
 \node[anchor=east] at (-.7,0) {$G:$};
\end{scope}
\begin{scope}[yshift=-1.2cm]
 \coordinate (x) at (.55,0);
 \coordinate (l1) at (0,.45); \coordinate (l2) at (0,-.45);
 \coordinate (y) at (7.45,0);
 \coordinate (r1) at (8,.45); \coordinate (r2) at (8,-.45);
 \draw (x)--(l1)--(l2)--cycle;
 \draw (y)--(r1)--(r2)--cycle;
 \draw (l1)--(-.4,.75) (l2)--(-.4,-.75)
       (r1)--(8.4,.75) (r2)--(8.4,-.75);
 \draw (x)--(y) node[midway,above=3pt] {$f=xy$};
 \foreach \p in {x,l1,l2,y,r1,r2} \node[vertex] at (\p) {};
 \node[below=3pt] at (x) {$x$}; \node[below=3pt] at (y) {$y$};
 \node[anchor=east] at (-.7,0) {$G_0:$};
\end{scope}
\draw[-{Stealth[length=2mm]}] (4,.55)--(4,-.45)
 node[midway,right=5pt] {replace the chain by one edge};
\end{tikzpicture}
\caption{Reduction in the base graph. The two triangle blocks and their
attachment vertices $x,y$ are retained. The diamond chain is replaced
by an edge $f=xy$, whose endpoints and chain length are recorded.
The initial colouring is constructed on $T(G_0)$; the diamonds are then
restored to obtain a colouring of $T(G)$. Unfinished edges continue to
the rest of the base.}
\label{fig:diamond-reduction}
\end{figure}
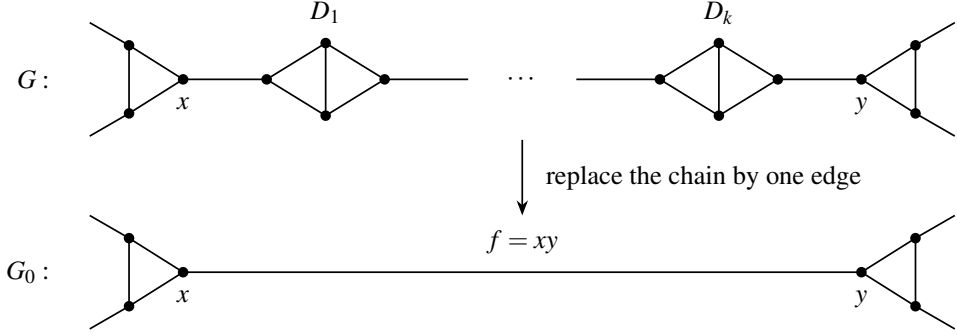

\FloatBarrier
\subsection{Colouring the truncation of the reduced base}\label{dt:section}

To colour the truncation of the reduced base from
Lemma~\ref{cf:reduced-base}, we work in the following more general
setting. Let $G_0$ be a simple cubic graph with a fixed partition
of its vertex set into triangles, and let $Q$ be obtained by
contracting them. Edges outside the triangles form a matching, so
$Q$ is loopless and cubic, and $G_0\cong T(Q)$. Parallel edges and
bridges in $Q$ are allowed. Brooks' theorem reduces each connected
component of $Q$ to the $3$-vertex-colourable case or the exceptional
case $K_4$; see the proof of Lemma~\ref{cf:reduced-colouring}.
We treat these cases separately, coordinating the boundary signatures
of the $9$-vertex pieces associated with the triangle blocks of $G_0$.

For each vertex $v$ of $Q$, let $X_v$ be the $9$-vertex piece of
$T(G_0)$ formed by truncating the vertices of the corresponding
triangle block of $G_0$. It contains three canonical triangles and
has three ports, one for each incidence at $v$ in $Q$. For an edge
$e=uv$ of $Q$, let $\iota(e)$ be the external edge of $T(G_0)$ arising
from the edge of $G_0$ corresponding to $e$. It joins a port
of $X_u$ to a port of $X_v$ and is called an \emph{interblock edge}.

Label the vertices of a model piece $X$ by $p_0,\ldots,p_8$, with
ports $p_0,p_5,p_8$ and edge set
\begin{equation}
\label{dt:gadget-edges}
\begin{split}
 &p_0p_1,\ p_0p_2,\ p_1p_2,\ p_1p_3,\ p_2p_4,\ p_3p_5,\\
 &p_3p_6,\ p_5p_6,\ p_4p_7,\ p_4p_8,\ p_7p_8,\ p_6p_7.
\end{split}
\end{equation}
The three canonical triangles are $p_0p_1p_2$, $p_3p_5p_6$ and
$p_4p_7p_8$. Attaching one pendant edge at each port gives the
\emph{augmented piece}; the pendant edges represent the interblock
edges before gluing. See Figure~\ref{dt:gadget-figure}.
Every permutation of the ports is induced by an automorphism of $X$,
obtained by permuting the vertices of the base triangle.
The pendant edges are pairwise nonconflicting, since their ports lie
in distinct canonical triangles and no edge of $X$ joins two ports.

\begin{figure}[htbp]
\centering
\begin{tikzpicture}[scale=.9,line width=.6pt,
  vertex/.style={circle,fill=black,inner sep=1.5pt},
  port/.style={circle,draw=black,fill=white,inner sep=2.1pt},
  every label/.style={font=\small}]
\coordinate (p0) at (0,3.2);
\coordinate (p1) at (-.55,2.25);
\coordinate (p2) at (.55,2.25);
\coordinate (p3) at (-1.35,.9);
\coordinate (p4) at (1.35,.9);
\coordinate (p5) at (-2.4,.9);
\coordinate (p6) at (-1.85,0);
\coordinate (p7) at (1.85,0);
\coordinate (p8) at (2.4,.9);
\draw (p0)--(p1)--(p2)--cycle;
\draw (p3)--(p5)--(p6)--cycle;
\draw (p4)--(p7)--(p8)--cycle;
\draw (p1)--(p3) (p2)--(p4) (p6)--(p7);
\draw[dashed] (p0)--(0,4) (p5)--(-3.2,.9) (p8)--(3.2,.9);
\node[port,label=above left:$p_0$] at (p0) {};
\node[vertex,label=left:$p_1$] at (p1) {};
\node[vertex,label=right:$p_2$] at (p2) {};
\node[vertex,label=right:$p_3$] at (p3) {};
\node[vertex,label=left:$p_4$] at (p4) {};
\node[port,label=above:$p_5$] at (p5) {};
\node[vertex,label=below:$p_6$] at (p6) {};
\node[vertex,label=below:$p_7$] at (p7) {};
\node[port,label=above:$p_8$] at (p8) {};
\end{tikzpicture}
\caption{The $9$-vertex piece $X$, with ports $p_0,p_5,p_8$ shown as open
circles. Dashed segments represent the three interblock edges.}
\label{dt:gadget-figure}
\end{figure}
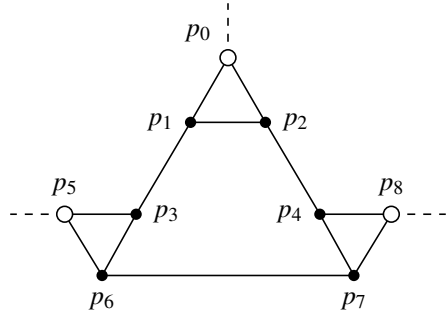

We use the signature notation of Section~\ref{pre:section} at the
level of pieces. For an incidence $(v,e)$ of $Q$, denote by $P_v(e)$
the exposed pair at the corresponding port of $X_v$, where the
subscript $v$ now identifies the piece. The pair consists of the colours
of the two internal edges incident with that port. It is attached to
the incidence, rather than to the colour of the interblock edge.
This leaves two tasks: realise suitable boundary prescriptions within
each piece, and coordinate those prescriptions across the edges of $Q$.
The following gluing statement is the local criterion of
Lemma~\ref{pre:interface} applied to whole pieces.

\Needspace{8\baselineskip}
\begin{lemma}
\label{dt:gluing}
Let $G_0$ be a simple cubic graph with a fixed partition
of its vertex set into triangles, and let $Q$ be obtained by
contracting these triangles. Let $X_v$ be the $9$-vertex piece
of $T(G_0)$ associated with $v\in V(Q)$. For each $v$, let $\varphi_v$ be a strong edge-colouring
of the augmented piece $X_v$ with colours in $\Omega$. Suppose that,
for every edge $e=uv$ of $Q$, the two pendant representatives of
$\iota(e)$ receive the same colour and their exposed pairs satisfy
\[
 P_u(e)\cap P_v(e)=\varnothing.
\]
Then the colourings $\varphi_v$ induce a strong edge-colouring of
$T(G_0)$ with colours in $\Omega$.
\end{lemma}

\begin{proof}
Each augmented piece has valid $6$-colour states and compatible
interfaces within the piece. Joining pendant representatives with the
same colour preserves all these local states. The only new interfaces
are the interblock edges, whose exposed pairs are disjoint by
hypothesis. Lemma~\ref{pre:interface} therefore gives the strong
colouring of $T(G_0)$.
\end{proof}

We only need pieces whose three interblock colours are all equal or
use exactly two colours. The following lemma gives boundary
prescriptions for these two patterns, denoted by $(a,a,a)$ and
$(a,a,b)$. Its colour-permutation argument allows the exposed pairs
to be chosen when the pieces are assembled.

\Needspace{10\baselineskip}
\begin{lemma}
\label{dt:realisability}
Let $X$ be the $9$-vertex piece with edge set~\eqref{dt:gadget-edges}
and ports $p_0,p_5,p_8$. Each of the following boundary prescriptions
extends to a strong edge-colouring of its augmented piece with
colours in $\Omega$.
\begin{enumerate}[label=(\roman*)]
\item For $a\in\{1,2,3\}$ and a $2$-element set
$P\subseteq\Omega\setminus\{a\}$, all three pendant edges have
colour $a$ and all three exposed pairs equal $P$.
\item For distinct $a,b\in\{1,2,3\}$ and
$\gamma\in\Omega\setminus\{a,b\}$, any prescribed two pendant
edges have colour $a$ and exposed pair $\{b,\gamma\}$, while the
remaining pendant edge has colour $b$ and exposed pair
$\{a,\gamma\}$.
\end{enumerate}
\end{lemma}

\begin{proof}
Table~\ref{dt:local-table} gives two base colourings.  The first has
interblock colours $(1,1,1)$ and exposed pair $\{4,5\}$ at every port.
The second has interblock colours $(2,1,1)$ at $(p_0,p_5,p_8)$, respectively;
its exposed pairs are $\{1,4\}$ at $p_0$ and $\{2,4\}$ at $p_5,p_8$.

\begin{table}[!htbp]
\centering
\begin{tabular}{ccc}
\toprule
Edge & Base $(1,1,1)$ & Base $(2,1,1)$ \\
\midrule
$p_0p_1$ & 4 & 1 \\
$p_0p_2$ & 5 & 4 \\
$p_1p_2$ & 6 & 6 \\
$p_1p_3$ & 2 & 3 \\
$p_2p_4$ & 3 & 5 \\
$p_3p_5$ & 5 & 4 \\
$p_3p_6$ & 3 & 5 \\
$p_5p_6$ & 4 & 2 \\
$p_4p_7$ & 2 & 3 \\
$p_4p_8$ & 4 & 2 \\
$p_7p_8$ & 5 & 4 \\
$p_6p_7$ & 6 & 6 \\
\midrule
Interblock edge at $p_0$ & 1 & 2 \\
Interblock edge at $p_5$ & 1 & 1 \\
Interblock edge at $p_8$ & 1 & 1 \\
\bottomrule
\end{tabular}
\caption{Two explicit strong colourings of the augmented $9$-vertex piece.}
\label{dt:local-table}
\end{table}

In either column, each canonical triangle and its three incident
external edges use all six colours.  Across the three edges joining different triangles,
the exposed pairs are, in the order $p_1p_3,p_2p_4,p_6p_7$,
\[
\begin{array}{c|ccc}
 &p_1p_3&p_2p_4&p_6p_7\\ \hline
(1,1,1)&\{4,6\},\{3,5\}&\{5,6\},\{2,4\}&\{3,4\},\{2,5\}\\
(2,1,1)&\{1,6\},\{4,5\}&\{4,6\},\{2,3\}&\{2,5\},\{3,4\}.
\end{array}
\]
Each displayed pair of sets is disjoint.  These checks account for
all strong conflicts in the augmented piece.

For a prescribed $(a,a,a)$ state, apply a colour permutation sending
$1$ to $a$ and $\{4,5\}$ to $P$ in the first column of
Table~\ref{dt:local-table}.
For a prescribed $(a,a,b)$ state, apply a permutation sending
$1$ to $a$, $2$ to $b$ and $4$ to $\gamma$ in the second column.
A port automorphism places its unique $b$-port at the required
incidence. Colour permutations and graph automorphisms preserve
the strong edge-colouring property, proving both assertions.
\end{proof}

Suppose first that $Q$ admits a proper $3$-vertex-colouring.
We use its vertex colours to choose both the interblock colours and
the exposed pairs. At each vertex, at most two interblock colours
occur, so every resulting prescription has one of the two patterns
covered by Lemma~\ref{dt:realisability}.

\begin{proposition}
\label{dt:three-vertex-colouring}
Let $G_0$ be a simple cubic graph with a fixed partition
of its vertex set into triangles, and let $Q$ be obtained by
contracting these triangles. If $Q$ admits a proper vertex-colouring
with three colours, then
\[
 \sidx(T(G_0))=6.
\]
\end{proposition}

\begin{proof}
Write the vertex colours as $R,B,V$. We first assign the interblock
colours, and then choose exposed pairs that make every join compatible.
Assign each edge of $Q$ its \emph{type} by
\[
 RB\longmapsto1,\qquad RV\longmapsto2,\qquad BV\longmapsto3.
\]
Give $\iota(e)$ the type of $e$. Although this need not be a proper
edge-colouring of $Q$, the interblock edges at any augmented piece are
pairwise nonconflicting.

For a vertex colour $C\in\{R,B,V\}$, the available interblock
types form the $2$-element set
\[
 S_R=\{1,2\},\qquad S_B=\{1,3\},\qquad S_V=\{2,3\}.
\]
The multiset of incident types is therefore either $(a,a,a)$ or
$(a,a,b)$.

For the exposed pairs, reserve one additional colour for each vertex
colour: set $\eta_R=4$, $\eta_B=5$, and $\eta_V=6$. Each pair will
contain this reserved colour and the other available interblock type.
Thus, if an incidence at a vertex of colour $C$ has type $c\in S_C$,
prescribe its exposed pair as $(S_C\setminus\{c\})\cup\{\eta_C\}$.
For pattern $(a,a,a)$ this is one common pair avoiding $a$.
For pattern $(a,a,b)$ it prescribes $\{b,\eta_C\}$ at the two
$a$-ports and $\{a,\eta_C\}$ at the $b$-port.  Thus every
prescription is realised by Lemma~\ref{dt:realisability}.

For each edge type, the prescriptions at its two endpoints are
\[
\begin{array}{c|c|c|c}
\text{type}&\text{endpoint colours}&\text{first exposed pair}&\text{second exposed pair}\\ \hline
1&R,B&\{2,4\}&\{3,5\}\\
2&R,V&\{1,4\}&\{3,6\}\\
3&B,V&\{1,5\}&\{2,6\}.
\end{array}
\]
The exposed pairs in each row are disjoint, so
Lemma~\ref{dt:gluing} gives the required colouring.

The lower bound follows from \eqref{pre:lowerbound}, applied to $G_0$.
\end{proof}

We now handle the remaining connected skeleton $Q=K_4$, for which
$G_0\cong T(K_4)$.

\begin{proposition}
\label{dt:k4-certificate}
Let $G_0=T(K_4)$. Then $T(G_0)$ admits a strong $6$-edge-colouring.
\end{proposition}

\begin{proof}
Label the vertices of $G_0$ by
$0,\ldots,11$ so that its four triangle blocks have vertex sets
\[
 \{0,1,2\},\quad\{3,4,5\},\quad\{6,7,8\},\quad\{9,10,11\},
\]
and the six edges between these triangles are
\[
 03,\quad16,\quad29,\quad47,\quad5\,10,\quad8\,11.
\]
For a vertex $q$ of $G_0$, list its neighbours in increasing order as
$(q_1,q_2,q_3)$. In Table~\ref{dt:k4-table}, the external and opposite
internal triples list the values of $c_q(qq_i)$ and $h_q(qq_i)$,
respectively, for $i=1,2,3$. This is the local-state notation of
Section~\ref{pre:section}, with the ports ordered by their neighbours.
Use these triples to colour $T(G_0)$.

\begin{table}[!htbp]
\centering
\begin{tabular}{c|c|c|c}
\hline
$q$ & $(q_1,q_2,q_3)$ & External & Opposite internal\\
\hline
0 & $(1,2,3)$ & $(2,3,1)$ & $(5,4,6)$\\
1 & $(0,2,6)$ & $(2,5,4)$ & $(6,1,3)$\\
2 & $(0,1,9)$ & $(3,5,6)$ & $(2,1,4)$\\
3 & $(0,4,5)$ & $(1,4,5)$ & $(6,3,2)$\\
4 & $(3,5,7)$ & $(4,2,6)$ & $(5,1,3)$\\
5 & $(3,4,10)$ & $(5,2,3)$ & $(6,1,4)$\\
6 & $(1,7,8)$ & $(4,1,5)$ & $(6,3,2)$\\
7 & $(4,6,8)$ & $(6,1,2)$ & $(5,4,3)$\\
8 & $(6,7,11)$ & $(5,2,6)$ & $(3,4,1)$\\
9 & $(2,10,11)$ & $(6,1,2)$ & $(4,5,3)$\\
10 & $(5,9,11)$ & $(3,1,4)$ & $(6,2,5)$\\
11 & $(8,9,10)$ & $(6,2,4)$ & $(3,5,1)$\\
\hline
\end{tabular}
\caption{A complete local-state certificate for $T(G_0)$ when $G_0=T(K_4)$.}
\label{dt:k4-table}
\end{table}

In every row, the two colour triples together contain exactly
$1,\ldots,6$. The external colours agree at both endpoints of every
edge of $G_0$, so the colouring is well-defined. Each exposed pair
$P_q(qq_i)$ is obtained by deleting $h_q(qq_i)$ from the opposite
internal triple.
All eighteen interfaces are listed below. A two-digit string $ij$
denotes the set $\{i,j\}$; an edge $q r$ refers to its two vertex labels.
\[
\begin{array}{c|cc@{\qquad}c|cc}
 e&P_q(e)&P_r(e)&e&P_q(e)&P_r(e)\\ \hline
 0\,1&46&13&0\,2&56&14\\
 0\,3&45&23&1\,2&36&24\\
 1\,6&16&23&2\,9&12&35\\
 3\,4&26&13&3\,5&36&14\\
 4\,5&35&46&4\,7&15&34\\
 5\,10&16&25&6\,7&26&35\\
 6\,8&36&14&7\,8&45&13\\
 8\,11&34&15&9\,10&34&56\\
 9\,11&45&13&10\,11&26&35
\end{array}
\]
Every displayed pair of sets is disjoint, so Lemma~\ref{pre:interface}
gives the required colouring.
\end{proof}

These two constructions give the following result, with no
connectivity or bridgelessness assumption on $G_0$.

\begin{lemma}\label{cf:reduced-colouring}
Let $G_0$ be a simple cubic graph whose vertex set has
a partition into triangles. Then
\[
 \sidx(T(G_0))=6.
\]
\end{lemma}
\begin{proof}
Contract the triangles to obtain a loopless cubic multigraph $Q$,
so that $G_0\cong T(Q)$. Consider one connected component of $Q$.
Its underlying simple graph $S$ is connected and has maximum degree
at most three. If $S\ne K_4$, Brooks' theorem \cite{Brooks1941},
together with the elementary case of maximum degree at most two,
gives a proper
$3$-vertex-colouring of $S$ and hence of the component of $Q$;
apply Proposition~\ref{dt:three-vertex-colouring}.
If $S=K_4$, cubicity forces the component of $Q$ itself to be $K_4$,
so apply Proposition~\ref{dt:k4-certificate}.
Reuse the same palette on all components. The lower bound in
\eqref{pre:lowerbound}, applied to $G_0$, gives equality.
\end{proof}

\FloatBarrier
\subsection{Inserting diamonds and completing the proof}

Return now to the reduced base $G_0$ of Lemma~\ref{cf:reduced-base}, and let
$\varphi$ be the colouring of $T(G_0)$ supplied by
Lemma~\ref{cf:reduced-colouring}. If $f=xy$ replaced a non-empty
diamond chain, its endpoint signatures have the form
\[
 (c_f,U_f),\qquad(c_f,V_f),\qquad
 U_f\cap V_f=\varnothing,
 \qquad c_f\notin U_f\cup V_f.
\]
The signatures are measured at the retained ports corresponding to
$x,y$. The new edges incident with these ports must keep colour $c_f$.
The diamond port attached to $x$ must expose a pair disjoint from
$U_f$, and similarly at $y$.

Regard a diamond with tips $a,b$ and centres $u,v$ as a cubic $2$-pole
by adjoining one dangling half-edge at each tip. Its \emph{$2$-port
truncation} is the augmented truncation of Section~\ref{pre:section}.
It consists of the four canonical triangles at $a,b,u,v$,
the five external edges corresponding to $au,av,bu,bv,uv$, and a
pendant edge at the outside port of each tip triangle. The pendant
endpoints are auxiliary and are not truncated. Boundary signatures
are measured at the outside ports of $\Delta_a$ and $\Delta_b$.

\begin{lemma}
\label{cf:diamond-boundary}
Let $D$ be a diamond with tips $a,b$ and centres $u,v$. Let
$c\in\Omega$, and let $U,V$ be disjoint $2$-element subsets of
$\Omega\setminus\{c\}$. Then the $2$-port truncation of $D$ admits
a strong edge-colouring with colours in $\Omega$ whose boundary signatures
are
\[
 (c,V)\quad\text{at }a,
 \qquad (c,U)\quad\text{at }b.
\]
\end{lemma}

\begin{proof}
Use the colouring in Table~\ref{cf:diamond-table}. In each row, the
external triple gives the colours at the ordered ports, and the
internal triple gives the colours of the opposite triangle edges.
The first ports at $a,b$ are the boundary ports.

\begin{table}[!htbp]
\centering
\begin{tabular}{c c c c}
\toprule
Vertex & Ordered ports & External & Opposite internal\\
\midrule
$a$ & $(\mathrm{out},u,v)$ & $(1,4,5)$ & $(6,2,3)$\\
$b$ & $(\mathrm{out},u,v)$ & $(1,2,3)$ & $(6,4,5)$\\
$u$ & $(a,b,v)$ & $(4,2,6)$ & $(3,5,1)$\\
$v$ & $(a,b,u)$ & $(5,3,6)$ & $(2,4,1)$\\
\bottomrule
\end{tabular}
\caption{A strong $6$-edge-colouring of the $2$-port truncation of a
diamond. The signatures at its two outside ports are
$(1,\{2,3\})$ and $(1,\{4,5\})$.}
\label{cf:diamond-table}
\end{table}

Each row contains all six colours. The five internal interfaces of
the $2$-port piece are checked explicitly below:
\[
\begin{array}{c|c|c|c}
\text{base edge}&\text{external colour}&\text{first exposed pair}&\text{second exposed pair}\\ \hline
au&4&\{6,3\}&\{5,1\}\\
av&5&\{6,2\}&\{4,1\}\\
bu&2&\{6,5\}&\{3,1\}\\
bv&3&\{6,4\}&\{2,1\}\\
uv&6&\{3,5\}&\{2,4\}
\end{array}
\]
The external colours agree and the exposed pairs are disjoint at
every interface, so Lemma~\ref{pre:interface} applies. The boundary
signatures are $(1,\{2,3\})$ at $a$ and $(1,\{4,5\})$ at $b$.

Choose a permutation of $\Omega$ taking $1$ to $c$, $\{2,3\}$ to $V$,
and $\{4,5\}$ to $U$, and apply it to the table. The resulting
colouring is strong and has the required signatures.
\end{proof}

The reversal of $U,V$ in the boundary prescription is what permits
iteration. A diamond ends with exposed pair $U$, while the next one
begins with $V$, so the same colouring can be repeated throughout the
chain. No alternation between different boundary states is needed.

\begin{lemma}\label{cf:diamond-chains}
Let $K$ be a loopless cubic multigraph, let $\varphi$ be a strong
edge-colouring of $T(K)$ with colours in $\Omega$, and let
$F\subseteq E(K)$. For each $e\in F$, choose a positive integer
$k_e$, and form $K'$ by replacing $e$ with a chain of $k_e$ diamonds.
Then $T(K')$ admits a strong edge-colouring with colours in $\Omega$
that agrees with $\varphi$ on every retained canonical triangle and
on every external edge corresponding to an edge of $E(K)\setminus F$.
\end{lemma}
\begin{proof}
For an edge $e=xy$ to be replaced, let $(c,U)$ and $(c,V)$ be its
signatures at $x$ and $y$ in $\varphi$. The interface criterion gives
$U\cap V=\varnothing$, with $U,V$ $2$-element subsets of
$\Omega\setminus\{c\}$.

Orient the replacing chain from $x$ to $y$. In each diamond, name its
entering tip $a$ and its leaving tip $b$, and apply
Lemma~\ref{cf:diamond-boundary} with the same data $c,U,V$. Every
diamond then has entering signature $(c,V)$ and leaving signature
$(c,U)$. All joining edges receive colour $c$.

At the first join the exposed pairs are $U$ at $x$ and $V$ at the
first diamond. Between consecutive diamonds the exposed pairs are
$U$ and $V$. At the last join they are $U$ at the final diamond and
$V$ at $y$. Thus all new interfaces are compatible. At either old
endpoint, its external edge still has colour $c$, so its local state
is unchanged.

The construction applies simultaneously to all replaced edges, even
when they share endpoints: preserving each old external colour keeps
every retained local state unchanged. Lemma~\ref{pre:interface} then
gives a strong colouring of $T(K')$. No condition on the length or
parity of any chain is used.
\end{proof}

Figure~\ref{fig:diamond-construction} shows the local colouring and its
repetition along a chain. Chains of length zero correspond to
unreplaced edges.

\begin{figure}[!htbp]
\centering
\begin{minipage}[t]{.51\textwidth}
\centering
\begin{minipage}[c][4.8cm][c]{\linewidth}
\centering
\begin{tikzpicture}[x=.8cm,y=.8cm,
  line width=.65pt,
  vertex/.style={circle,fill=black,inner sep=1.45pt},
  auxiliary/.style={circle,draw=black,fill=white,inner sep=1.45pt},
  colour/.style={font=\small,text=red!70!black,fill=white,inner sep=1.5pt},
  every node/.style={font=\small}]
  \coordinate (ao) at (-3.55,0);
  \coordinate (au) at (-2.5,.45);
  \coordinate (av) at (-2.5,-.45);
  \coordinate (bo) at (3.55,0);
  \coordinate (bu) at (2.5,.45);
  \coordinate (bv) at (2.5,-.45);
  \coordinate (ua) at (-.55,1.8);
  \coordinate (ub) at (.55,1.8);
  \coordinate (uv) at (0,1.2);
  \coordinate (va) at (-.55,-1.8);
  \coordinate (vb) at (.55,-1.8);
  \coordinate (vu) at (0,-1.2);
  \draw (au)--(av) node[midway,right=3pt,colour] {$6$};
  \draw (ao)--(av) node[midway,below=2pt,colour] {$2$};
  \draw (ao)--(au) node[midway,above=2pt,colour] {$3$};
  \draw (bu)--(bv) node[midway,left=3pt,colour] {$6$};
  \draw (bo)--(bv) node[midway,below=2pt,colour] {$4$};
  \draw (bo)--(bu) node[midway,above=2pt,colour] {$5$};
  \draw (ub)--(uv) node[midway,right=2pt,colour] {$3$};
  \draw (ua)--(uv) node[midway,left=2pt,colour] {$5$};
  \draw (ua)--(ub) node[midway,above=2pt,colour] {$1$};
  \draw (vb)--(vu) node[midway,right=2pt,colour] {$2$};
  \draw (va)--(vu) node[midway,left=2pt,colour] {$4$};
  \draw (va)--(vb) node[midway,below=2pt,colour] {$1$};
  \draw (au)--(ua) node[midway,above left,colour] {$4$};
  \draw (av)--(va) node[midway,below left,colour] {$5$};
  \draw (bu)--(ub) node[midway,above right,colour] {$2$};
  \draw (bv)--(vb) node[midway,below right,colour] {$3$};
  \draw (uv)--(vu) node[midway,right=3pt,colour] {$6$};
  \draw (-4.45,0)--(ao) node[midway,above=2pt,colour] {$1$};
  \draw (bo)--(4.45,0) node[midway,above=2pt,colour] {$1$};
  \foreach \p in {ao,au,av,bo,bu,bv,ua,ub,uv,va,vb,vu}
    \node[vertex] at (\p) {};
  \node[auxiliary] at (-4.45,0) {};
  \node[auxiliary] at (4.45,0) {};
  \node at (-2.99,0) {$\Delta_a$};
  \node at (2.99,0) {$\Delta_b$};
  \node at (0,2.65) {$\Delta_u$};
  \node at (0,-2.65) {$\Delta_v$};
  \node[anchor=north] at (-3.7,-.65) {$(1,\{2,3\})$};
  \node[anchor=north] at (3.7,-.65) {$(1,\{4,5\})$};
\end{tikzpicture}
\end{minipage}
\par\smallskip
{\footnotesize (a) A $2$-port diamond colouring\par}
\end{minipage}\hfill
\begin{minipage}[t]{.47\textwidth}
\centering
\begin{minipage}[c][4.8cm][c]{\linewidth}
\centering
\begin{tikzpicture}[x=1cm,y=1cm,
  line width=.65pt,
  block/.style={draw,rounded corners=2pt,minimum height=.7cm,
    minimum width=1.05cm,fill=white},
  endblock/.style={draw,rounded corners=2pt,minimum height=.7cm,
    minimum width=1.05cm,inner sep=2pt,fill=white},
  every node/.style={font=\footnotesize}]
  \node[endblock,align=center] (L) at (0,0) {retained\\piece};
  \node[block] (D1) at (1.7,0) {$D_1$};
  \node[block] (Dk) at (4.1,0) {$D_k$};
  \node[endblock,align=center] (R) at (5.8,0) {retained\\piece};
  \draw (L.east)--(D1.west) node[midway,above=4pt] {$c$};
  \draw (D1.east)--(2.5,0) node[midway,above=4pt] {$c$};
  \node at (2.9,0) {$\cdots$};
  \draw (3.3,0)--(Dk.west) node[midway,above=4pt] {$c$};
  \draw (Dk.east)--(R.west) node[midway,above=4pt] {$c$};
  \node[below=7pt,xshift=.10cm] at (L.east) {$U$};
  \node[below=7pt,xshift=-.10cm] at (D1.west) {$V$};
  \node[below=7pt,xshift=.10cm] at (D1.east) {$U$};
  \node[below=7pt,xshift=-.10cm] at (Dk.west) {$V$};
  \node[below=7pt,xshift=.10cm] at (Dk.east) {$U$};
  \node[below=7pt,xshift=-.10cm] at (R.west) {$V$};
  \node at (2.9,-.92)
    {Every $D_i$: entry $(c,V)$, exit $(c,U)$};
\end{tikzpicture}
\end{minipage}
\par\smallskip
{\footnotesize (b) The same boundary state on every diamond\par}
\end{minipage}
\caption{(a) The $2$-port diamond colouring of Table~\ref{cf:diamond-table}.
Red numbers denote edge colours; open circles are auxiliary degree-$1$
endpoints. (b) Already coloured pieces of $T(G_0)$ are joined through
the restored diamond pieces. Each diamond has entry signature $(c,V)$ and exit signature $(c,U)$, so every join
compares the disjoint pairs $U,V$, regardless of the chain length.
For $k=1$ there is one diamond; for $k=0$ the end blocks are joined directly.}
\label{fig:diamond-construction}
\end{figure}
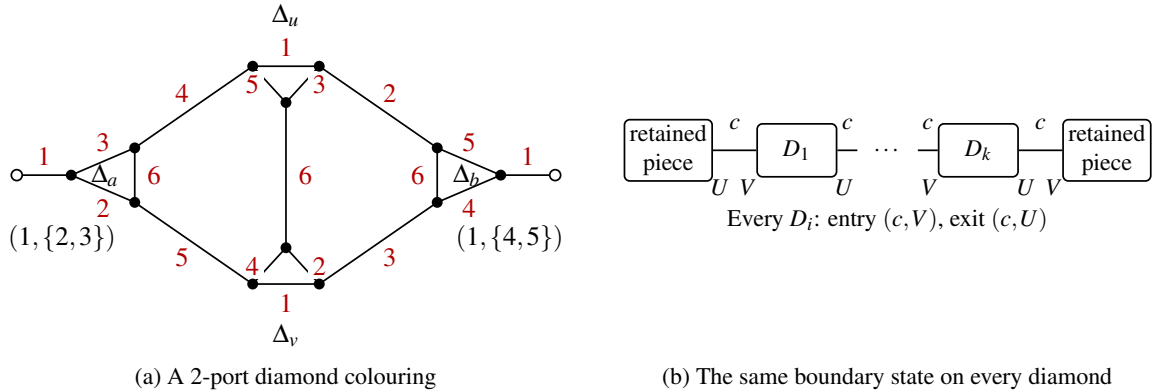

\FloatBarrier
\Needspace{9\baselineskip}
The reduction to triangle blocks leaves two cases: the original
base $G=K_4$ and the graphs whose blocks are all diamonds. Both can
be closed using the same boundary data as above.
In the all-diamond case we close a chain into a cycle; for $K_4$ we
join the two tips of a single diamond.

\begin{lemma}
\label{cf:finite-initial}
Let $G$ be a connected simple bridgeless cubic claw-free graph.
Suppose that either $G=K_4$, or $G\ne K_4$ and every block in the
partition of Lemma~\ref{cf:blocks} is a diamond. Then $T(G)$ admits
a strong edge-colouring with colours in $\Omega$.
\end{lemma}

\begin{proof}
In the all-diamond case the contracted multigraph $J$ is a cycle by
Lemma~\ref{cf:skeleton}. Orient this cycle; for a length-$2$ cycle,
its two parallel edges have opposite directions. In each diamond,
call the entering tip $a$ and the leaving tip $b$. Colour every diamond
by Table~\ref{cf:diamond-table}, without changing the colours. The
entering signature is $(1,\{2,3\})$ and the leaving signature is
$(1,\{4,5\})$. At every join the external colour is $1$ and the exposed
pairs are disjoint, including at the last join closing the cycle.
The interface criterion therefore gives a strong $6$-edge-colouring.

For $K_4$, remove the edge $ab$ to view its four vertices as the diamond
with tips $a,b$ and centres $u,v$. Use the same table and join the two
dangling ports by the external edge corresponding to $ab$, with colour
$1$. Their exposed pairs $\{2,3\}$ and $\{4,5\}$ are disjoint. This
checks the sixth base edge in addition to the five already checked in
Lemma~\ref{cf:diamond-boundary}, so the resulting colouring of $T(K_4)$
is strong. Here $K_4$ is represented by a diamond together with its missing tip
edge; it is not an induced-diamond block.
\end{proof}

\begin{proof}[Proof of Theorem~\ref{cf:main}]
We combine the structural reduction with the two colouring constructions.
If $G=K_4$, use the closed-diamond colouring in
Lemma~\ref{cf:finite-initial}. Suppose therefore that $G\ne K_4$ and
form its block partition by Lemma~\ref{cf:blocks}. If all blocks are
diamonds, the cycle construction of Lemma~\ref{cf:finite-initial}
colours $T(G)$.

Otherwise, replace every diamond chain by an edge to obtain the
reduced base $G_0$. Keep the attachment vertices and the length of each
replaced chain. By Lemma~\ref{cf:reduced-base}, we have $G_0\cong T(Q)$ for a connected
loopless cubic multigraph $Q$.

Colour $T(G_0)$ by Lemma~\ref{cf:reduced-colouring}.
Recover $G$ from $G_0$ by restoring its
recorded diamond chains. Lemma~\ref{cf:diamond-chains} extends the
colouring to $T(G)$ and preserves every canonical triangle retained
from $T(G_0)$. Chains of length zero require no replacement. This proves
$\sidx(T(G))\leq6$. The reverse inequality follows from
\eqref{pre:lowerbound}.
\end{proof}

Lemma~\ref{cf:reduced-colouring} applies whenever the base has a
partition into triangles. Taking the base to be $T(Q)$ therefore gives
the following corollary, including when $Q$ has bridges.

\begin{corollary}\label{dt:multigraph-extension}
Let $Q$ be a loopless cubic multigraph. Then
\[
 \sidx(T(T(Q)))=6.
\]
\end{corollary}
\begin{proof}
The graph $G=T(Q)$ is simple and cubic, and its canonical
triangles partition its vertex set. Lemma~\ref{cf:reduced-colouring}
therefore gives $\sidx(T(G))=\sidx(T(T(Q)))=6$, regardless of
whether $Q$ has bridges.

When $Q$ is bridgeless, $G$ is also bridgeless and claw-free.
In that case the conclusion follows directly from
Theorem~\ref{cf:main}, applied to each connected component of $G$.
\end{proof}

\FloatBarrier
\section{Multigraph extensions}\label{ext:section}

We now consider bases that may have parallel edges.
Corollary~\ref{bip:multigraph-nondet} already settles the non-det-extremal
bipartite case. We first study the boundary behaviour of a block formed
by two parallel edges, use it to colour cyclic arrangements of such
blocks and diamonds, and deduce a general chain insertion result.
We then extend the rooted replacement construction to obtain further
bipartite examples. Section~\ref{concl:section} explains which cases
these constructions settle and what remains open. Throughout this
section, all multigraphs are loopless, $\Omega=\{1,2,3,4,5,6\}$,
and parallel edges and their incidences are distinguished.

\subsection{Parallel-edge blocks and cyclic constructions}

Let $B$ be the $2$-pole consisting of two vertices $u,v$ joined by
two parallel edges $p,q$, with one dangling edge at each vertex.
Its boundary behaviour differs from that of a diamond: the two
boundary edges must have different colours. The following proposition
determines all boundary prescriptions realised by its augmented
truncation.

\begin{proposition}\label{ext:double-rule}
The pairs of boundary signatures of strong edge-colourings of the
augmented truncation of $B$ with colours in $\Omega$ are exactly
\begin{equation}\label{concl:double-rule}
 (c,W),\quad(d,W),\qquad
 c\ne d,\quad |W|=2,\quad W\cap\{c,d\}=\varnothing,
\end{equation}
where $c,d\in\Omega$ and $W\subseteq\Omega$.
\end{proposition}
\begin{proof}
For sufficiency, write $W=\{r,s\}$ and
$\Omega=\{c,d,a,b,r,s\}$. In the port order
$(\mathrm{out},p,q)$, use the states
\[
 u:(c,a,b;d,r,s),\qquad v:(d,a,b;c,s,r).
\]
The exposed pairs at the boundary ports both equal $W$, and the
pairs at each parallel-edge interface are disjoint. Thus
Lemma~\ref{pre:interface} verifies the colouring.

For necessity, let $a,b$ be the colours on the two parallel edges.
Equal boundary colours would give identical internal palettes at
$u,v$, preventing disjoint exposed pairs on either parallel edge.
For distinct boundary colours $c,d$, the four colours $a,b,c,d$ are
distinct. Writing the remaining colours as $r,s$, the internal
palettes are $\{d,r,s\}$ and $\{c,r,s\}$. At each parallel-edge
interface, the two opposite internal colours must cover the common
colours $r,s$. Since the opposite internal colours at each vertex
are distinct, these two interfaces use $r,s$ at each vertex, leaving
$d,c$ opposite the boundary ports. Thus both boundary exposed pairs
equal $\{r,s\}$, proving~\eqref{concl:double-rule}.
\end{proof}

We next combine these blocks with diamonds in cyclic arrangements.
The only exception is the closure of a single block $B$, which gives
the $3$-bond $M_0$.

\begin{proposition}\label{ext:block-cycles}
Let $G$ be obtained by joining a non-empty cyclic sequence of disjoint
$2$-poles, each of which is a copy of $B$ or a diamond with one
dangling edge at each tip. If $G\not\cong M_0$, then $\sidx(T(G))=6$.
\end{proposition}
\begin{proof}
First suppose that the cycle contains $k\ge2$ blocks $B$. Replace
any intervening diamond chains by single edges, leaving a cycle of
these $k$ blocks. Partition the palette into $U=\{1,2\}$,
$V=\{3,4\}$ and $Z=\{5,6\}$. For even $k$,
alternate the exposed pairs $U,V$ on the blocks and the colours
$5,6$ on the joining edges. For odd $k$, alternate $U,V$ on the
first $k-1$ blocks and use $Z$ on the last. Alternate $5,6$ on the
first $k-2$ joins, and use $1$ on the join from the final $V$-block
to the $Z$-block and $3$ on the join from the $Z$-block to the first
$U$-block. In both cases~\eqref{concl:double-rule} applies at every
block, and the exposed pairs at every join are disjoint. Any diamonds
can then be reinserted by Lemma~\ref{cf:diamond-chains}.

If the cycle contains exactly one block $B$ and at least one
diamond, replacing its diamond chain by a single edge would give
$M_0$, so the preceding construction does not apply. Instead, start from the following
$6$-vertex base. Let the diamond have tips $a,b$
and centres $u,v$, let $p,q$ be the parallel edges $xy$, and add
$ax,by$. A $6$-colour certificate is
\[
\begin{array}{c|c|c|c}
\text{vertex}&\text{port order}&\text{external colours}&
 \text{opposite internal colours}\\ \hline
a&(u,v,x)&(1,2,3)&(4,5,6)\\
b&(u,v,y)&(2,6,4)&(5,3,1)\\
u&(a,b,v)&(1,2,5)&(6,3,4)\\
v&(a,b,u)&(2,6,5)&(4,1,3)\\
x&(p,q,a)&(6,5,3)&(2,1,4)\\
y&(p,q,b)&(6,5,4)&(1,2,3).
\end{array}
\]
Each row uses all six colours. Across each base edge, the external
colours agree and the two exposed pairs are disjoint. Thus
Lemma~\ref{pre:interface} verifies the certificate, and
Lemma~\ref{cf:diamond-chains} inserts any further diamonds.
Cycles consisting only of diamonds, including the closure
giving $K_4$, were treated in Lemma~\ref{cf:finite-initial}.
The only remaining arrangement, one block $B$ with its two ports
joined, is the excluded graph $M_0$. In every other case the local
lower bound~\eqref{pre:lowerbound} gives equality.
\end{proof}

Cutting a joining edge in a coloured cyclic arrangement gives an
insertion construction. The two ends of the cut edge have the same
external colour and disjoint exposed pairs, so the gluing argument
of Lemma~\ref{bip:edge-sum} matches them to any coloured edge of
another base.

\begin{corollary}\label{ext:double-chains}
Let $K$ be a loopless cubic multigraph, let $\varphi$ be a strong
edge-colouring of $T(K)$ with colours in $\Omega$, and let
$e\in E(K)$. Form $K'$ by replacing $e$ with a non-empty chain of
copies of $B$ and diamonds. If the chain does not consist of a single
copy of $B$, then $T(K')$ admits a strong edge-colouring with colours
in $\Omega$ that agrees with $\varphi$ on every retained canonical
triangle and on every external edge corresponding to an edge of
$E(K)\setminus\{e\}$.
\end{corollary}
\begin{proof}
Join the two free ends of the chain by an edge $f$, obtaining a
loopless cubic multigraph $F$. Since the chain is not a single
copy of $B$, we have $F\not\cong M_0$, and
Proposition~\ref{ext:block-cycles} gives a strong $6$-edge-colouring
of $T(F)$. The graph $K'$ is the edge-sum of $K$ and $F$ along
$e$ and $f$. Apply the gluing construction in the proof of
Lemma~\ref{bip:edge-sum}, permuting colours only on $T(F)$ and
assigning the original external colour of $e$ to both new joining
edges. This preserves all local states on $K$ and all other
external colours, as required.
\end{proof}

The excluded chain cannot be inserted while preserving both endpoint
states: a single block $B$ requires distinct boundary colours by
Proposition~\ref{ext:double-rule}, whereas the replaced edge has
the same colour at both ends. In particular, the corollary covers a
chain containing just one block $B$ as soon as it also contains a
diamond.

\subsection{Rooted replacements and bipartite cycles}

We now return to det-extremal bipartite bases. The rooted construction
depends on matching signs and on the three incidences at each vertex.
Distinct neighbours are not required, so the same construction gives
an exact vertex simulator when parallel edges are present.

\begin{proposition}\label{ext:rooted-simulator}
Let $H$ be a connected loopless cubic bipartite det-extremal
multigraph, and let $v\in V(H)$. Then the $3$-pole $H-v$ is an
exact vertex simulator.
\end{proposition}
\begin{proof}
Choose partition sets $A,B$ with $v\in B$ and follow the construction
of Theorem~\ref{bip:rooted-simulator}, distinguishing parallel
incidences. As explained in the proof of
Corollary~\ref{bip:multigraph-nondet}, the
digraph $D_M$ is still strongly connected, and det-extremality implies
that every directed cycle is odd. Loops are omitted from the outgoing
spanning arborescence. A loop at its root closes the tree path of
length zero, so the parity argument giving~\eqref{bip:root-bits}
remains valid. The resulting assignment has mixed neighbourhoods at
every vertex of $B\setminus\{v\}$ and gives bit zero to every neighbour
of $v$.
The three edge types distinguish the ports of $v$ even when neighbours
coincide. The colouring formula and the final colour permutation
therefore give every prescribed ordered triple of signatures in
\eqref{bip:vertex-signatures}.
\end{proof}

This simulator allows larger bipartite pieces to take the place of
the blocks $B$ in a cyclic arrangement. For an edge $e$ of a cubic
multigraph $H$, write $H-e$ for the $2$-pole retaining a dangling
half-edge at each endpoint of the deleted edge. When $H$ is connected,
bipartite and det-extremal, this pole realises every boundary
prescription needed for a block $B$.

\begin{corollary}\label{ext:det-pole-cycles}
Let $k\ge2$. For each $i\in\{1,\ldots,k\}$, let $H_i$ be a
connected loopless cubic bipartite det-extremal multigraph and let
$e_i\in E(H_i)$. If $G$ is obtained by joining disjoint copies of
the $2$-poles $H_i-e_i$ cyclically, then $\sidx(T(G))=6$.
\end{corollary}
\begin{proof}
For any connected loopless cubic bipartite det-extremal multigraph
$H$ and any edge $e$ incident with $z$, the $2$-pole $H-e$ realises every boundary
prescription in~\eqref{concl:double-rule}. Indeed, start with a
colouring of $B$ having the prescribed signatures and replace one
of its vertices by the exact vertex simulator $H-z$ from
Proposition~\ref{ext:rooted-simulator}. Assign the port corresponding
to $e$ to the outside port of the replaced vertex. The unreplaced
vertex of $B$ takes the role of $z$: its two non-boundary edges
restore the other two incidences of $z$. The resulting pole is
therefore $H-e$, with a dangling half-edge at each endpoint of $e$.
The replacement preserves the prescribed boundary signatures by
Lemma~\ref{pre:interface}.

Apply this observation to every $H_i-e_i$, using the boundary
prescriptions for a cycle of $k$ blocks $B$ in the proof of
Proposition~\ref{ext:block-cycles}. All joins are compatible, so
Lemma~\ref{pre:interface} supplies the colouring. The lower bound
follows from~\eqref{pre:lowerbound}.
\end{proof}

\FloatBarrier
\section{Concluding remarks}\label{concl:section}

Theorems~\ref{bip:main} and~\ref{cf:main} establish
$\sidx(T(G))=6$ for connected simple cubic bases that are
bipartite or are both bridgeless and claw-free. The constructions
above also settle several multigraph cases, but do not yet give
the full extension of either theorem to parallel edges.

Double truncations are covered without a bridgelessness assumption:
Corollary~\ref{dt:multigraph-extension} gives
$\sidx(T(T(Q)))=6$ for every loopless cubic multigraph $Q$,
including disconnected graphs and graphs with bridges. Here the base
of the final truncation is the simple graph $T(Q)$. The question for
a single truncation $T(Q)$ of a multigraph base is different; the
following discussion records what our results establish in that
setting.

\subsection{Bipartite multigraphs}

The non-det-extremal case is completely covered by
Corollary~\ref{bip:multigraph-nondet}: every connected loopless cubic
bipartite multigraph $G$ that is not det-extremal satisfies
$\sidx(T(G))=6$. Both the binary assignment and the colouring
formula remain valid when parallel incidences are distinguished.

For det-extremal bases, Proposition~\ref{ext:rooted-simulator}
extends the exact vertex simulator $H-v$ to multigraphs, and
Corollary~\ref{ext:det-pole-cycles} colours cyclic arrangements of
the $2$-poles $H_i-e_i$. These constructions give compatible
replacement pieces and colourings of specific families. They do
not yet colour $T(H)$ for every det-extremal multigraph
$H\not\cong M_0$.

The remaining difficulty in the proof via $2$-edge cuts is that an
end closure may be the $3$-bond $M_0$, whose truncation requires
nine colours. Thus the proof of Theorem~\ref{bip:main} cannot always
colour the two closures separately with six colours and then join
them. This is an obstruction in the auxiliary closures, not a
counterexample on the original truncation.

If a simple, $3$-connected and det-extremal end closure $P$ is
available, the alternating-cycle argument in the proof of
Lemma~\ref{bip:end-piece} still works: it gives an alternating cycle
of length divisible by four in the complementary closure $R$, even
when $G$ and $R$ have parallel edges. Hence $R$ is non-det-extremal,
and Corollary~\ref{bip:multigraph-nondet} colours $T(R)$. The edge-sum
construction can then join this colouring to that of $T(P)$.
Decompositions whose available end closures are $3$-bonds require
a further argument.

\subsection{Claw-free multigraphs}

For multigraphs, claw-freeness refers to the underlying simple graph.
In a connected bridgeless cubic claw-free multigraph other than
$M_0$, each pair of parallel edges forms a block $B$. Its two outside
neighbours are distinct, since a common neighbour would give a
bridge. Thus these blocks are disjoint from triangles, and the
triangle--diamond decomposition applies to the remaining vertices,
with $K_4$ handled separately.

If there are no triangle blocks, the graph is a cyclic arrangement
of diamonds and blocks $B$, so Proposition~\ref{ext:block-cycles}
settles this case completely. If triangle blocks are present, the
same contraction argument as in Lemma~\ref{cf:skeleton} partitions
the remaining blocks into chains between distinct triangle blocks:
a chain returning to the same triangle would leave its third
incidence on a bridge. Replacing these chains by edges gives a
reduced base $G_0\cong T(Q)$ for a loopless cubic multigraph $Q$.
In particular, $G_0$ is simple and has a partition into triangles,
so Lemma~\ref{cf:reduced-colouring} colours $T(G_0)$.

Corollary~\ref{ext:double-chains} then restores every non-empty
chain except one consisting of a single block $B$, preserving all
local states of the retained triangles. Consequently, the claw-free
extension is proved whenever none of the chains between triangle
blocks consists of a single $B$ without diamonds. Chains with one
block $B$ and at least one diamond are already covered.

Only the insertion of a single block $B$ remains outside this
construction. By Proposition~\ref{ext:double-rule}, its boundary
colours must be distinct, so it cannot replace an edge while
preserving both endpoint states. This local incompatibility does
not show that the resulting truncation needs more than six colours.
The unresolved step is to coordinate the boundary states at chains
consisting of a single block $B$, allowing the retained pieces to be
recoloured.

\subsection{The general question}

The general problem remains open:
\begin{quote}
\itshape
Does every connected loopless bridgeless cubic multigraph
$G\not\cong M_0$ satisfy $\sidx(T(G))=6$?
\end{quote}
The exclusion of $M_0$ is necessary: $T(M_0)$ is the triangular
prism and $L(T(M_0))^2=K_9$. The hypothesis of bridgelessness
cannot be dropped either, since the multigraph underlying Tanwar's
example has a bridge and its truncation has strong chromatic index
seven~\cite{Tanwar2026}.

The question also includes simple bases outside the two classes
treated here. Within these classes, the remaining tasks are the
det-extremal bipartite extension beyond the constructions above and
the insertion of a single block $B$ between triangle blocks in the
claw-free case.

\FloatBarrier
\Needspace{16\baselineskip}

\end{document}